\documentclass[11pt,reqno]{amsart}

\allowdisplaybreaks

\usepackage{amsmath, amsfonts, amssymb, amsthm, xcolor, hyperref, tikz, verbatim}
\usepackage[capitalise]{cleveref}

\newtheorem{theorem}{Theorem}[section]
\newtheorem{lemma}[theorem]{Lemma}
\newtheorem{proposition}[theorem]{Proposition}
\newtheorem{corollary}[theorem]{Corollary}

\theoremstyle{definition}
\newtheorem{definition}[theorem]{Definition}
\newtheorem{example}[theorem]{Example}
\newtheorem{remark}[theorem]{Remark}
\newtheorem{definition-and-remark}[theorem]{Definition and Remark}
\newtheorem{notation}[theorem]{Notation}
\newtheorem{notation-and-remark}[theorem]{Notation and Remark}
\newtheorem{remark-and-notation}[theorem]{Remark and Notation}
\newtheorem{remark-and-definition}[theorem]{Remark and Definition}
\newtheorem{ad-hoc-item}[theorem]{}

\newcommand{\bC}{ \mathbb{C} }
\newcommand{\bE}{ \mathbb{E} }
\newcommand{\bN}{ \mathbb{N} }

\newcommand{\bR}{ \mathbb{R} }
\newcommand{\bZ}{ \mathbb{Z} }

\newcommand{\area}{\mathrm{area}}

\newcommand{\InT}{ \mathrm{Int} }
\newcommand{\OuT}{ \mathrm{Out} }
\newcommand{\Area}{A}
\newcommand{\NCord}{ NC^{\mathrm{(mton)}} }
\newcommand{\parent}{ \mathfrak{p} }
\newcommand{\pairparent}{ \mathfrak{pp} }
\newcommand{\CPair}{C^{ (\mathrm{pair})} }
\newcommand{\pp}{r}
\newcommand{\ur}{\underline{r}}

\title[Statistics on monotonically ordered non-crossing 
partitions]{Statistics on monotonically ordered \\
non-crossing partitions}

\author[N. Blitvic]{Natasha Blitvic}
\address{Natasha Blitvic: School of Mathematical Sciences, 
         Queen Mary University of London, London UK.} 
\thanks{NB: research supported by grant EP/V048902 from EPSRC, UK}
\email{n.blitvic@qmul.ac.uk}

\author[T. Bray]{Thomas Bray}
\address{Thomas Bray: Department of Pure Mathematics, 
         University of Waterloo, Ontario, Canada.} 
\email{tgrbray@uwaterloo.ca}

\author[J. Campbell]{Jacob Campbell}
\address{Jacob Campbell: Department of Mathematics, 
University of Virginia, Charlottesville, Virginia, USA.}
\email{cgh6uv@virginia.edu}

\author[A. Nica]{Alexandru Nica}
\thanks{AN: research supported by a Discovery Grant from 
	NSERC, Canada.}
\address{Alexandru Nica: Department of Pure Mathematics, 
	University of Waterloo, Ontario, Canada.}
\email{anica@uwaterloo.ca}

\begin{document}

\begin{abstract}
We study some combinatorial statistics defined on the set $\NCord (n)$ of 
monotonically ordered non-crossing partitions of $\{ 1, \ldots , n \}$, and 
on the set $\NCord_2 (2n)$ of monotonically ordered non-crossing 
pair-partitions of $\{ 1, \ldots , 2n \}$. 
Unlike in the analogous results known for unordered non-crossing partitions, the 
computations of expectations and variances for natural block-counting statistics 
on $\NCord (n)$ and for the expectation of the area statistic on $\NCord_2 (2n)$ 
turn out to yield a logarithmic regime.  An important role in our study is played
by a nice tree structure on the disjoint union of the $\NCord (n)$'s, which we use 
to streamline our arguments.  As an illustration of how these ideas can be applied 
to calculations of cumulants in monotone probability, we discuss some combinatorial 
aspects of the monotonic Poisson process.

\end{abstract}

\maketitle

\section{Introduction}

\subsection{\texorpdfstring{\boldmath{$NC(n)$}}{NC(n)} 
and \texorpdfstring{\boldmath{$\NCord (n)$}}{NCord(n)}.
The \texorpdfstring{\boldmath{$\NCord$}}{NCord}-tree structure.}

$\ $

\noindent
We will use the customary notation $NC(n)$ for the set of non-crossing partitions
of $\{ 1, \ldots , n \}$.  The elements of $NC(n)$ are thus of the form 
$\pi = \{ V_1, \ldots , V_k \}$, where $V_1, \ldots , V_k$ (the {\em blocks} of the 
partition $\pi$) are non-empty pairwise disjoint subsets of $\{ 1, \ldots , n \}$, 
with $V_1 \cup \cdots \cup V_k = \{ 1, \ldots , n \}$, and where it is not possible
to find $a_1 < a_2 < a_3 < a_4$ in $\{ 1, \ldots , n \}$ such that 
$a_1, a_3 \in V_i$ and $a_2, a_4 \in V_j$ for some $i \neq j$.  The $NC(n)$'s are one 
of the many combinatorial structures counted by Catalan numbers.  For a quick introduction 
to some of their basic properties, one can e.g.~consult \cite[Lecture 9]{NiSp2006}.

This paper is concerned with sets of 
{\em monotonically ordered} non-crossing partitions.

\vspace{6pt}

\begin{definition}
Let $n$ be a positive integer.

\noindent
$1^o$ Let $\pi = \{ V_1, \ldots , V_k \} \in NC(n)$.  One has a nesting relation 
among the blocks of $\pi$, where ``$V_i$ nested inside $V_j$'' means, by definition, 
that $\min (V_i) > \min (V_j)$ and $\max (V_i) < \max (V_j)$.  We will 
use the term {\em monotonic ordering of $\pi$} to refer to a bijection 
$u : \pi \to \{ 1, \ldots , k \}$ having the property that
$\bigl[ \mbox{$V_i$ nested inside $V_j$} \bigr] \, \Rightarrow
\, \bigl[ \, u(V_i) > u(V_j) \, \bigr].$  

\noindent
$2^o$ We denote
$\NCord (n) := \bigl\{ ( \pi, u ) 
\mid \pi \in NC(n) \mbox{ and $u$ is a monotonic ordering of } \pi \bigr\}$.
\end{definition}

\vspace{6pt}

\begin{remark}   \label{rem:12}
For $( \pi, u ) \in \NCord (n)$, one can think of $u$ as ``a possible history 
of how $\pi$ was created, via successive insertions of intervals''.  In order to 
clarify what this means, let us write explicitly $\pi = \{ V_1, \ldots , V_k \}$ 
and let us note that the block $V_{j_o} := u^{-1} (k) \in \pi$ is an interval of 
$\{ 1, \ldots , n \}$. 
(If not, then there would exist $i \neq j_o$ in $\{ 1, \ldots , k \}$ such that 
$V_i$ is nested inside $V_{j_o}$, hence such that $u( V_i ) > u( V_{j_o} ) = k$, 
in contradiction with $u( V_i ) \in \{ 1, \ldots , k \}$.)  One can view the
insertion of the interval $V_{j_o}$ as the last step in a $k$-step process which 
created $\pi$, via successive insertions of intervals which eventually became the 
blocks $u^{-1} (1), \ldots , u^{-1} (k)$, in this order.
For illustration, Figure 1 shows how this process works to create a monotonically 
ordered partition with $4$ blocks in $NC(9)$.

On a heuristic level, one could say that the relation between a $( \pi, u ) \in \NCord (n)$ 
and the underlying $\pi \in NC(n)$ is akin to the relation one has, in the representation 
theory of symmetric groups, between a standard Young tableau and the underlying Young diagram
-- the tableau records a possible history of how the Young diagram was created, via addition 
of boxes.  From this perspective, the $\NCord$-tree that we will next look at is an analogue 
of the tree of monotonic paths of the well-known Young graph (as presented, for instance, in 
Section 3.1 of the monograph \cite{BoOl2017}).
\end{remark}

\begin{comment}
\begin{figure}[t]
    \begin{tikzpicture}[scale=0.5]
        \draw (0,0) -- (0,1) -- (3,1) -- (3,0);
        \draw (1,0) -- (1,1);
        \draw (2,0) -- (2,1);

        \node at (4,0.5) {$\rightarrow$};

        \draw (5,0) -- (5,1.5) -- (9,1.5) -- (9,0);
        \draw (6,0) -- (6,1.5);
        \draw (7,0) -- (7,1.5);
        \draw (8,0) -- (8,1);

        \node at (10,0.5) {$\rightarrow$};

        \draw (11,0) -- (11,1.5) -- (17,1.5) -- (17,0);
        \draw (12,0) -- (12,1.5);
        \draw (13,0) -- (13,1) -- (14,1) -- (14,0);
        \draw (15,0) -- (15,1.5);
        \draw (16,0) -- (16,1);

        \node at (18,0.5) {$\rightarrow$};

        \draw (19,0) -- (19,1) -- (20,1) -- (20,0);
        \draw (21,0) -- (21,1.5) -- (27,1.5) -- (27,0);
        \draw (22,0) -- (22,1.5);
        \draw (23,0) -- (23,1) -- (24,1) -- (24,0);
        \draw (25,0) -- (25,1.5);
        \draw (26,0) -- (26,1);
    \end{tikzpicture}
    \label{fig:1}
    \caption{Building an element of $NC(9)$ by inserting intervals}
\end{figure}
\end{comment}

\begin{figure}[t]
    \begin{tikzpicture}[scale=0.5]
        \draw[line width=1.5pt] (0,0.1) -- (0,-0.9) -- (3,-0.9) -- (3,0.1);
        \draw[line width=1.5pt] (1,0.1) -- (1,-0.9);
        \draw[line width=1.5pt] (2,0.1) -- (2,-0.9);
        \node at (0.3,-0.4) { \bf{1} };

        \node at (4,-0.5) {$\rightarrow$};

        \draw (5,0) -- (5,-1.5) -- (9,-1.5) -- (9,0);
        \draw (6,0) -- (6,-1.5);
        \draw (7,0) -- (7,-1.5);
        \draw[line width=1.5pt] (8,0.1) -- (8,-0.9);
        \node at (5.3,-1) {$1$};
        \node at (8.3,-0.4) { \bf{2} };

        \node at (10,-0.5) {$\rightarrow$};

        \draw (11,0) -- (11,-1.5) -- (17,-1.5) -- (17,0);
        \draw (12,0) -- (12,-1.5);
        \draw[line width=1.5pt] (13,0.1) -- (13,-0.9) -- (14,-0.9) -- (14,0.1);
        \draw (15,0) -- (15,-1.5);
        \draw (16,0) -- (16,-1);
        \node at (11.3,-1) {$1$};
        \node at (13.3,-0.4) { \bf{3} };
        \node at (16.3,-0.5) {$2$};

        \node at (18,-0.5) {$\rightarrow$};

        \draw[line width=1.5pt] (19,0.1) -- (19,-0.9) -- (20,-0.9) -- (20,0.1);
        \draw (21,0) -- (21,-1.5) -- (27,-1.5) -- (27,0);
        \draw (22,0) -- (22,-1.5);
        \draw (23,0) -- (23,-1) -- (24,-1) -- (24,0);
        \draw (25,0) -- (25,-1.5);
        \draw (26,0) -- (26,-1);
        \node at (19.3,-0.4) { \bf{4} };
        \node at (21.3,-1) {$1$};
        \node at (23.3,-0.5) {$3$};
        \node at (26.3,-0.5) {$2$};
    \end{tikzpicture}
    \label{fig:1}
    \caption{ {\em $\pi = \bigl\{ \, \{1,2\}, \, \{3,4,7,9\},
               \, \{5,6\}, \, \{8\} \, \bigr\} \in NC(9)$ built in 4 steps, 
            by inserting intervals.  The building process gives a monotonic ordering $u$ 
            of $\pi$: every block of $\pi$ retains the label assigned to it (in 
            boldface font) at the moment when it was inserted in the picture, thus giving
            $u ( \, \{3,4,7,9\} \, ) = 1$,  $u ( \, \{8\} \, ) = 2$, 
            $u ( \, \{5,6\} \, ) = 3$,  $u ( \, \{1,2\} \, ) = 4$.} }
\end{figure}
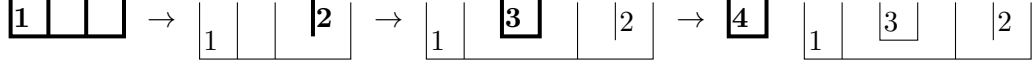

% \vspace{6pt}

\begin{remark}  \label{rem:13}
{\em (The $\NCord$-tree.)} 
In the considerations of the present paper, a crucial role will be played 
by a tree structure with vertex set
$\NCord := \sqcup_{n=1}^{\infty} \NCord (n)$ (disjoint union), where the 
unique element of $\NCord (1)$ serves as root and where, for every $n \in \bN$, 
the elements of $\NCord (n)$ are at distance $n-1$ from the root.
The formal description of the edges of this $\NCord$-tree is given in 
Section \ref{subsection:2-1} below.  Roughly speaking, it is based on the
interpretation of every $( \pi, u) \in \NCord (n)$ as possible history for 
how the underlying $\pi = \{ V_1, \ldots , V_k \} \in NC(n)$ was created, 
but where we amend the $k$-step creation process described in 
Remark \ref{rem:12}, and make it become an {\em $n$-step} creation process 
(use the number $n$ of points that are being partitioned, rather than the 
number $k$ of blocks of $\pi$).
That is: instead of ``intervals inserted from above'' as in Figure 1, we will now 
look at insertions of individual points, as shown in Figure 2.  At every step, the 
new incoming point: either starts a new block of the partition which is 
being built, or appends itself to the highest-numbered block (necessarily an 
interval block) that already existed in the picture.  Specifically, Figure 2 shows 
the same partition $\pi \in \NCord (9)$ as in Figure 1, where the construction of 
$\pi$ now takes $9$ steps (rather than $4$ steps, as before).
\end{remark}

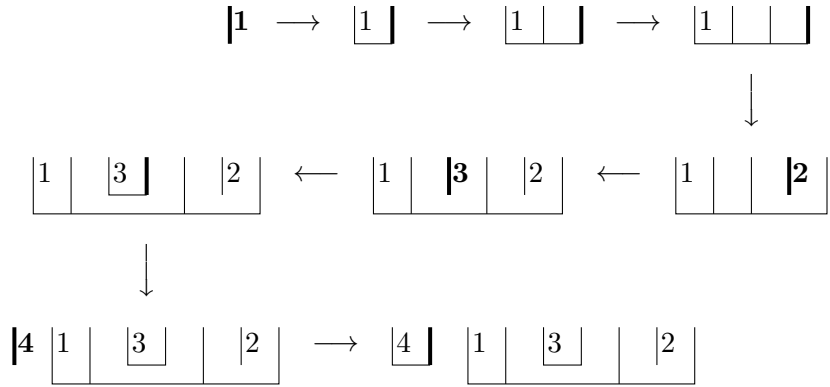
\begin{figure}[b]
    \begin{tikzpicture}[scale=0.5]
        \draw[line width=1.5pt] (-0.3,0) -- (-0.3,-1);
        \node at (0,-0.4) { \bf{1} };
        
        \node at (1.5,-0.5) {$\longrightarrow$};

        \draw (3,0) -- (3,-1) -- (4,-1);
        \draw[line width=1.5pt] (4,-1) -- (4,0);
        \node at (3.3,-0.4) {$1$};

        \node at (5.5,-0.5) {$\longrightarrow$};

        \draw (7,0) -- (7,-1) -- (9,-1);
        \draw (8,0) -- (8,-1);
        \draw[line width=1.5pt] (9,0) -- (9,-1);
        \node at (7.3,-0.4) {$1$};

        \node at (10.5,-0.5) {$\longrightarrow$};

        \draw (12,0) -- (12,-1) -- (15,-1);
        \draw[line width=1.5pt] (15,0) -- (15,-1);
        \draw (13,0) -- (13,-1);
        \draw (14,0) -- (14,-1);
        \node at (12.3,-0.4) {$1$};

        \node at (13.5,-2.5) {$\Big\downarrow$};

        \draw (11.5,-4) -- (11.5,-5.5) -- (15.5,-5.5) -- (15.5,-4);
        \draw (12.5,-4) -- (12.5,-5.5);
        \draw (13.5,-4) -- (13.5,-5.5);
        \draw[line width=1.5pt] (14.5,-4) -- (14.5,-5);
        \node at (11.8,-4.4) {$1$};
        \node at (14.8,-4.4) { \bf{2} };

        \node at (10,-4.5) {$\longleftarrow$};

        \draw (8.5,-4) -- (8.5,-5.5) -- (3.5,-5.5) -- (3.5,-4);
        \draw (7.5,-4) -- (7.5,-5);
        \draw (6.5,-4) -- (6.5,-5.5);
        \draw[line width=1.5pt] (5.5,-4) -- (5.5,-5);
        \draw (4.5,-4) -- (4.5,-5.5);
        \node at (3.8,-4.4) {$1$};
        \node at (5.8,-4.4) { \bf{3} };
        \node at (7.8,-4.4) {$2$};

        \node at (2,-4.5) {$\longleftarrow$};

        \draw (0.5,-4) -- (0.5,-5.5) -- (-5.5,-5.5) -- (-5.5,-4);
        \draw (-4.5,-4) -- (-4.5,-5.5);
        \draw (-3.5,-4) -- (-3.5,-5) -- (-2.5,-5);
        \draw[line width=1.5pt] (-2.5,-4) -- (-2.5,-5);
        \draw (-1.5,-4) -- (-1.5,-5.5);
        \draw (-0.5,-4) -- (-0.5,-5);
        \node at (-5.2,-4.4) {$1$};
        \node at (-3.2,-4.4) {$3$};
        \node at (-0.2,-4.4) {$2$};

        \node at (-2.5,-7) {$\Big\downarrow$};

        \draw[line width=1.5pt] (-6,-8.5) -- (-6,-9.5);
        \draw (-5,-8.5) -- (-5,-10) -- (1,-10) -- (1,-8.5);
        \draw (-4,-8.5) -- (-4,-10);
        \draw (-3,-8.5) -- (-3,-9.5) -- (-2,-9.5) -- (-2,-8.5);
        \draw (-1,-8.5) -- (-1,-10);
        \draw (0,-8.5) -- (0,-9.5);
        \node at (-5.7,-8.9) { \bf{4} };
        \node at (-4.7,-8.9) {$1$};
        \node at (-2.7,-8.9) {$3$};
        \node at (0.3,-8.9) {$2$};

        \node at (2.5,-9) {$\longrightarrow$};

        \draw (4,-8.5) -- (4,-9.5) -- (5,-9.5);
        \draw[line width=1.5pt] (5,-8.5) -- (5,-9.5);
        \draw (6,-8.5) -- (6,-10) -- (12,-10) -- (12,-8.5);
        \draw (7,-8.5) -- (7,-10);
        \draw (8,-8.5) -- (8,-9.5) -- (9,-9.5) -- (9,-8.5);
        \draw (10,-8.5) -- (10,-10);
        \draw (11,-8.5) -- (11,-9.5);
        \node at (4.3,-8.9) {$4$};
        \node at (6.3,-8.9) {$1$};
        \node at (8.3,-8.9) {$3$};
        \node at (11.3,-8.9) {$2$};
    \end{tikzpicture}
    
    \label{fig:2}
    \caption{ {\em $\pi = \bigl\{ \, \{1,2\}, \, \{3,4,7,9\},
               \, \{5,6\}, \, \{8\} \, \bigr\} \in NC(9)$ built in 9 steps,  
            by inserting points.  The building process gives the path in the $\NCord$-tree which 
            starts at the root of the tree and arrives at the monotonically ordered non-crossing 
            partition from Figure 1.} }
\end{figure}

\vspace{1cm}

\begin{remark}   \label{rem:14}
{\em (Homogeneity property of the $\NCord$-tree.)}  
One has that:
\begin{equation}   \label{eqn:14a}
\left\{  \begin{array}{c}
\mbox{for every $n \in \bN$, every
$( \pi , u ) \in \NCord (n)$}   \\
\mbox{has precisely $n+2$ children in the $\NCord$-tree.}
\end{array}  \right.
\end{equation}
This is a fact that is easy to verify (cf.~Remark \ref{rem:25} below),
but nevertheless has great benefits towards the study of block-counting 
statistics on $\NCord (n)$, because it provides recursions satisfied by 
such statistics.  A number of results obtained in this way are described
in the next subsection.

We note that, as a byproduct, the homogeneity from (\ref{eqn:14a}) can be used
as an entry point towards some non-commutative probability considerations.  It is 
relevant to mention here that $NC(n)$ and $\NCord (n)$ play crucial roles in the 
combinatorics of two brands of independence for non-commutative random variables, 
{\em free independence} and respectively {\em monotonic independence}.  In particular, 
each of these of two brands of independence uses a notion of {\em cumulants} for the 
non-commutative random variables that are considered, where the cumulants of free 
probability are defined (see e.g.~\cite[Lecture 11]{NiSp2006}) with the help of the 
$NC(n)$'s, while the cumulants of monotonic probability are defined (cf.~\cite{HaSa2011}) 
with the help of the $\NCord (n)$'s.  
The homogeneity (\ref{eqn:14a}) can be used in order to obtain recursion formulas for 
monotonic cumulants.  The present paper does not aim to detail this direction,
but there is one immediate example that felt worth describing,
concerning the recursion for the monotonic cumulants of the so-called ``monotonic Poisson
process'' -- cf. Section \ref{section:5} below, where we succinctly review the monotonic
Poisson process, and use (\ref{eqn:14a}) to retrieve some results from \cite{Be2006} about
the cumulants of the said process.
\end{remark}

\vspace{10pt}

\subsection{Block-counting statistics 
on \texorpdfstring{\boldmath{$\NCord (n)$}}{NCord(n)}.}
\label{subsection:1-2}

$\ $

\noindent
For motivation, we mention that there exists a fair amount of research literature
(see e.g.~\cite{Ar2012}, \cite{Or2012}, or the more recent presentation made in 
\cite{Ka2020}) which is
dedicated to the study of block-statistics on $NC(n)$, with $NC(n)$ viewed as 
a uniform probability space.  The calculations made in this regime often yield 
quantities that, asymptotically, grow linearly with $n$.  For instance if 
$X_n : NC(n) \to \bN$ is the random variable that counts blocks,
$X_n ( \pi ) := | \pi | = k$ for $\pi = \{ V_1, \ldots , V_k \} \in NC(n)$, then 
the expectation and variance of $X_n$ are found to be
\begin{equation}   \label{eqn:1-2a}
E[X_n] = \frac{n+1}{2} \ \mbox{ and } 
\ \mathrm{Var} [X_n] = \frac{n^2 -1}{4(2n-1)} \sim \frac{n}{8} 
\end{equation}
(see e.g.~\cite[Theorem 2.1]{Ka2020}).  Another relevant result is that upon 
fixing an $\ell \in \bN$ and considering the random variables 
$X_n^{( \ell )} : NC(n) \to \bN \cup \{ 0 \}$ which are defined by
\[
X_n^{( \ell )} ( \pi ) := 
\ \vline \, \{ V \in \pi \mid \, |V| = \ell \} \ \vline \ ,
\ \ \pi \in NC (n),
\]
one gets that $E[ X_n^{ ( \ell ) } ] \sim n/2^{\ell + 1}$ for $n \to \infty$
(see e.g.~\cite[Corollary 4.3]{Or2012}).

In the present paper we look at the analogous block-counting random variables,
in the parallel world of $\NCord (n)$.  

\vspace{6pt}

\begin{remark-and-notation}   \label{rem:15}
$1^o$ For every $n \in \bN$, we view $\NCord (n)$ as a uniform probability space,
where every $( \pi, u ) \in \NCord (n)$ has probability $1/ | \NCord (n) |$. 
The latter cardinality is $| \NCord (n) | =  (n+1)!/2$, 
as had been noted in the literature related to monotonic probability
(cf.~\cite[Proposition 3.4]{ArHaLeVa2015}), and can also be easily derived from the 
homogeneity property (\ref{eqn:14a}) of the $\NCord$-tree.

\vspace{6pt}

\noindent
$2^o$ For every $n \in \bN$, we let $Y_n : \NCord (n) \to \bN$ be the
random variable defined by 
\begin{equation}   \label{eqn:1-2b}
Y_n ( \pi , u ) := | \pi | 
\mbox{ (number of blocks of $\pi$),} \ \ ( \pi, u ) \in \NCord (n).
\end{equation}
The next theorem gives the expectation and variance of $Y_n$.  Unlike the formulas 
(\ref{eqn:1-2a}) concerning the $X_n$'s, the formulas for the $Y_n$'s show a 
logarithmic regime.  In the statements of the latter formulas we will use the 
standard notation
\[
H_n := 1 + \frac{1}{2} + \cdots + \frac{1}{n}, 
\ \mbox{ ($n$-th harmonic number, for $n \in \bN$),}
\]
and $\gamma := \lim_{n \to \infty} H_n - \ln n$ (the Euler-Mascheroni constant).
\end{remark-and-notation}

\vspace{6pt}

\begin{theorem}   \label{thm:16}
Let the random variables $Y_n : \NCord (n) \to \bN$ be 
as defined in (\ref{eqn:1-2b}). 

\vspace{6pt}

\noindent
$1^o$ One has 
$E[ \, Y_n \, ] = n - H_n + \frac{3}{2} - \frac{1}{n+1}, \ \ n \geq 2.$
As a consequence, it follows that
\begin{equation}   \label{eqn:16b}
E[ \, Y_n \, ] \approx (n - \ln n) + \bigl( \, \frac{3}{2} - \gamma \, \bigr)
\mbox{ for } n \to \infty,
\end{equation}
where the meaning of ``$\approx$'' in (\ref{eqn:16b}) is that the 
difference of the quantities on its two sides goes to $0$ for $n \to \infty$.

\vspace{6pt}

\noindent
$2^o$ One has $\mathrm{Var} [Y_n] = H_n - H_n^{(2)} - \frac{(n-1)^2}{4(n+1)^2}, 
\ \ n \geq 2,$
where $H_n^{(2)} := 1 + \frac{1}{2^2} + \cdots + \frac{1}{n^2}.$
As a consequence, it follows that 
\begin{equation}  \label{eqn:16d}
\mathrm{Var} [Y_n] \approx 
\ln n - \bigl( \frac{\pi^2}{6} + \frac{1}{4} - \gamma \bigr),
\mbox{ for $n \to \infty$}.
\end{equation}
\end{theorem}

\vspace{6pt}

Suppose next that, for randomly chosen $( \pi, u ) \in \NCord (n)$, 
one tallies the blocks of $\pi$ according to their cardinality.  The next
theorem says that, for large $n$, one should expect $\pi$ to have a lot 
of blocks of size $1$ (about $n - 2 \, \ln n$ of them), a fairly significant 
number of blocks of size $2$ (about $\ln n$ of them), and very few blocks of 
size $\geq 3$.  

\vspace{6pt}

\begin{theorem}   \label{thm:17} 
For every $n, \ell \in \bN$, let $Y_n^{( \ell )} : \NCord (n) \to \bN \cup \{ 0 \}$
be defined by
\begin{equation}   \label{eqn:17a}
Y_n^{( \ell )} ( \pi, u )  
:= \ \vline \, \{ V \in \pi \mid \, |V| = \ell \} \ \vline \ ,
\ \ ( \pi, u ) \in \NCord (n).
\end{equation}
There exist constants $c_1, c_2, \ldots , c_{\ell}, \ldots \in \bR$ such that:
\begin{equation}   \label{eqn:17b}
\left\{  \begin{array}{c}
\lim_{n \to \infty} E[Y_n^{(1)}] - (n - 2 \, \ln n) = c_1, 
\ \lim_{n \to \infty} E[Y_n^{(2)}] - \ln n = c_2,   \\
\mathrm{and} \ \lim_{n \to \infty} E[Y_n^{( \ell )}] = c_{\ell}, 
\ \ \forall \, \ell \geq 3. 
\end{array}  \right.
\end{equation}
\end{theorem}

\vspace{6pt}

\begin{remark}   \label{rem:18}
$1^o$ Proposition \ref{prop:46} below gives the precise values 
$c_1 = \frac{10}{3} - 2 \gamma$ and $c_2 = \gamma - \frac{17}{8}$.  
The values of $c_{\ell}$ for $\ell \geq 3$ can also be explicitly computed,
but there doesn't appear to be some nice general formula for them (see discussion 
in Remark \ref{rem:47}.2 below).

\vspace{6pt}

\noindent
$2^o$ In the setting of the preceding theorem, one can in fact bundle together all 
the random variables $Y_n^{ ( \ell ) }$ with $\ell \geq 3$ into  
$Y_n^{ ( \geq 3 ) } : \NCord (n) \to \bN \cup \{ 0 \}$, where
\begin{equation}   \label{eqn:18a}
Y_n^{( \geq 3 )} ( \pi, u )  
:= \ \vline \, \{ V \in \pi \mid \, |V| \geq 3 \} \ \vline \ ,
\ \ ( \pi, u ) \in \NCord (n).
\end{equation}
It is still the case that the expectations 
$E[ \, Y_n^{(\geq 3)} \, ]$ have a finite limit for $n \to \infty$. 
Indeed, by writing $Y_n^{ ( \geq 3 ) } = Y_n - Y_n^{(1)} - Y_n^{(2)}$,
with $Y_n$ picked from Theorem \ref{thm:16}, one immediately finds that 
$\lim_{n \to \infty} E[ \, Y_n^{(\geq 3)} \, ] 
= ( \frac{3}{2} - \gamma ) - c_1 - c_2 = \frac{7}{24}$.
\end{remark}

\vspace{6pt}

\begin{remark}   \label{rem:19}
We reiterate the fact that the formulas announced in Theorems \ref{thm:16} and
\ref{thm:17} are derived by starting from the homogeneity property (\ref{eqn:14a}) 
of the $\NCord$-tree.  As an intermediate between the observation from 
(\ref{eqn:14a}) and explicit formulas for expectations and variances, we will put 
into evidence two kinds of recursions that may be satisfied by a sequence of random 
variables $( Z_n : \NCord (n) \to \bR )_{n=1}^{\infty}$.

\vspace{6pt}

\noindent
-- The recursion of the first kind uses for input a tuple 
$\ur = (r_1, \ldots , r_k) \in \bR^k$, for some $k \in \bN$.  We will introduce it 
in Section \ref{subsection:2-2}, then in Section \ref{section:4} we will show how 
this kind of recursion is used to approach the random variables $Y_n$ and 
$Y_n^{ ( \ell ) }$ from Theorems \ref{thm:16} and \ref{thm:17} (both these theorems 
are consequences of the same general result about the recursion of the first kind, 
given in Theorem \ref{thm:32}).

\vspace{6pt}

\noindent
-- The recursion of the second kind uses for input a triple 
$( \alpha, \beta ; q )$, with $\alpha, \beta, q \in \bZ$.
We will introduce it in Section \ref{subsection:2-3}, and we will discuss 
it further in Section \ref{section:6}.  This recursion of the second kind 
can be applied to the statistics counting outer blocks of $\pi$, in 
connection to a randomly picked $( \pi, u ) \in \NCord (n)$.  A block $V$ 
of a partition $\pi \in NC(n)$ is said to be {\em outer} when it 
is not possible to find $W \in \pi$ such that $V$ is nested 
inside $W$.  If $Z_n : \NCord (n) \to \bN$ is defined by putting
$Z_n ( \pi, u) :=
\ \vline \, \{ V \in \pi \mid \, V \mbox{ is outer} \} \ \vline \ ,$ 
then (as we will see in Section \ref{section:6}, by invoking the recursion of 
the second kind) one has the neat formula
$E[ Z_n ] = (2n+1)/3, \ \ \forall \, n \in \bN.$
\end{remark}

\vspace{10pt}

\subsection{The 
\texorpdfstring{\boldmath{$\NCord_2$}}{NCord2}-tree,
and some block-counting statistics on 
\texorpdfstring{\boldmath{$\NCord_2 (2n)$}}{NCord2(2n)}. }
\label{subsection:1-3}

$\ $

\noindent
In this paper we will also look at
some combinatorial statistics on monotonically ordered non-crossing
{\em pair-partitions}.  For motivation we mention that
the sets of unordered non-crossing pair-partitions,
$NC_2 (2n) := \bigl\{ \pi = \{ V_1, \ldots , V_n \} \in NC(2n) \ \vline 
\ \ |V_1| = \cdots = |V_n| = 2 \bigr\},$
play an important role in the combinatorics of free probability,
primarily due to their occurrence in the free 
Central Limit Theorem (see e.g. \cite[Lecture 8]{NiSp2006}).
This carries through to monotonic probability, where the sets
\begin{equation}   \label{eqn:1-3a}
\NCord_2 (2n) := \Bigl\{ ( \pi , u ) 
\begin{array}{lr}
\vline  & \pi \in NC_2 (2n)
\mbox{ and } u : \pi \to \{1, \ldots , n\},  \\
\vline  & \mbox{monotonic ordering} \end{array} \Bigr\}
\end{equation}
appear in the combinatorial description of the monotonic 
Central Limit Theorem \cite{Mu2001, HaSa2011}.

Towards studying combinatorial statistics on the $\NCord_2 (2n)$'s, one 
can pursue a methodology similar to the one described in 
Remark \ref{rem:19} in connection to the $\NCord (n)$'s.  This is because
the disjoint union $\NCord_2 := \sqcup_{n=1}^{\infty} \NCord_2 (2n)$ bears 
a natural structure of rooted tree analogous to (and in fact easier to spot than)
the one of the $\NCord$-tree of Remark \ref{rem:13}.  The unique pairing in 
$\NCord_2 (2)$ serves as root for the $\NCord_2$-tree and, for every $n \in \bN$, 
the elements of $\NCord_2 (2n)$ are at distance $n-1$ from the root.  Moreover, 
the $\NCord_2$-tree has a homogeneity property analogous to the one recorded in 
(\ref{eqn:14a}), which now goes like this:
\begin{equation}   \label{eqn:1-3b}
\left\{  \begin{array}{c}
\mbox{for every $n \in \bN$, every
$( \pi , u ) \in \NCord_2 (2n)$}   \\
\mbox{has precisely $2n+1$ children in the $\NCord_2$-tree.}
\end{array}  \right.
\end{equation}
We note that, as an immediate consequence of (\ref{eqn:1-3b}), one has 
\begin{equation}   \label{eqn:1-3bb}
| \, \NCord_2 (2n) \, | = (2n-1)!!, \ \ \forall \, n \in \bN
\end{equation}
(showing, in particular, how normalizations have to be done in calculations 
with uniform random variables on $\NCord_2 (2n)$).

Two interesting block-counting statistics that can be considered on 
$\NCord_2 (2n)$ are
\begin{equation}   \label{eqn:1-3c}
\OuT_n : \NCord_2(2n) \to \bR \ \mbox{ and }
\ \InT_n : \NCord_2(2n) \to \bR,
\end{equation}
counting the {\em outer pairs} and respectively the 
{\em interval pairs} of $\pi$, for a random $( \pi, u ) \in \NCord_2 (2n)$. 
Such pairs represent the top and bottom of the natural nesting structure among 
the pairs of $\pi$: a pair $V \in \pi$ is outer when it is not nested inside 
any other pair, and is an \emph{interval pair} when it is of the form 
$V = \{ m, m+1 \}$ for some $1 \leq m \leq 2n-1$ (which is equivalent to 
requiring that no $U \in \pi$ is nested inside $V$).  In Section \ref{section:7} 
of the paper we show how the recursion ``of the second kind'' from 
Remark \ref{rem:19} can be adapted to the setting of the 
$\NCord_2$-tree, and used to address the random variables $\OuT_n, \InT_n$ 
from (\ref{eqn:1-3c}).  In particular, one gets nice explicit formulas for 
their expectations, as follows.

\vspace{6pt}

\begin{theorem}   \label{thm:110}
Let $\InT_n, \OuT_n : \NCord_2 (2n) \to \bR$ be as defined above.

\vspace{6pt}

\noindent 
$1^o$ For every $n \in \bN$, one has that 
$E[ \, \InT_n \, ] = (2n+1)/3$.
\noindent 

\vspace{6pt}

\noindent
$2^o$ One has
\begin{equation}  \label{eqn:110a}
E[ \, \OuT_n \, ] = \frac{2^n \cdot n!}{ (2n-1)!! } - 1, \ \ n \in \bN,
\end{equation}
with the consequence that
$E[ \, \OuT_n \, ] \sim \sqrt{\pi n}$ for $n \to \infty$.
\end{theorem}

% \vspace{10pt}

\subsection{Another combinatorial statistic 
on \texorpdfstring{\boldmath{$\NCord_2 (2n)$}}{NCord2(2n)}: area.}
\label{subsection:1-4}

$\ $

\noindent
An interesting statistic, specific to non-crossing pair-partitions, is the one 
which measures the {\em area} under a $\pi \in NC_2 (2n)$, when $\pi$ is viewed 
in its incarnation as a Dyck lattice path (see e.g.~\cite[Lecture 2]{NiSp2006}) 
or, in calculus language, when $\pi$ is identified to a piecewise linear function 
$f_{\pi} : [0, 2n] \to [0, \infty )$.  To be precise, 
the function $f_{\pi}$ is determined by the following requirements:
\begin{equation}   \label{eqn:1-5a}
\left\{   \begin{array}{l}
\mbox{$\to$ $f_{\pi}$ is continuous on $[0, 2n]$, 
      and has $f_{\pi} (0) = 0$;} \\
\mbox{$\to$ $f_{\pi}$ is linear on every interval 
      $[m-1, m]$, $1 \leq m \leq 2n$;} \\
\mbox{$\to$ For every $V = \{ i,j \} \in \pi$, with $i < j$, the slope 
of $f_{\pi} \mid [i-1,i]$}  \\
\mbox{ \hspace{0.3cm} is equal to $1$ and the slope of 
$f_{\pi} \mid [j-1,j]$ is equal to $-1$.}
\end{array}  \right.
\end{equation}
The recipe used to define $f_{\pi}$ in (\ref{eqn:1-5a}) has the 
immediate consequences that $f_{\pi} (t) \geq 0$ for every $t \in [0, 2n]$
and that $f_{\pi} (2n) = 0$.  It is natural to put
\[
\mathrm{area} ( \pi ) := 
\int_0^{2n} f_{\pi} (t) \, dt \in [0, \infty )
\]
(where the latter quantity is easily seen, in fact, to always be a positive 
integer).
For illustration, Figure 3 below shows how the graph of $f_{\pi}$
looks like, for an example of pair-partition $\pi \in NC_2 (14)$.

\[
\hspace{1.2cm}
\pi = \begin{tikzpicture}[scale=0.5,baseline={2*height("$=$")}]
        \draw (0,0) -- (0,1.5) -- (5,1.5) -- (5,0);
        \draw (1,0) -- (1,1) -- (2,1) -- (2,0);
        \draw (3,0) -- (3,1) -- (4,1) -- (4,0);
        \draw (6,0) -- (6,1) -- (7,1) -- (7,0);
        \draw (8,0) -- (8,2) -- (13,2) -- (13,0);
        \draw (9,0) -- (9,1.5) -- (12,1.5) -- (12,0);
        \draw (10,0) -- (10,1) -- (11,1) -- (11,0);
\end{tikzpicture} \in NC_2(14) 
\]

$\ $

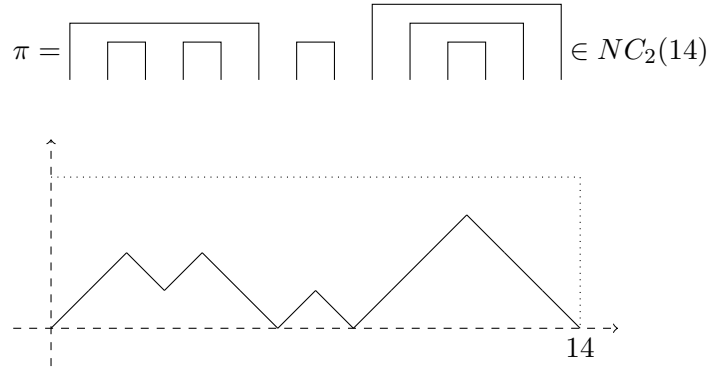
\begin{figure}[h]
\begin{tikzpicture}[scale=0.5]
            \draw[dashed,->] (-1,0) -- (15,0);
            \draw[dashed,->] (0,-1) -- (0,5);
            \draw (0,0) -- (1,1);
            \draw (1,1) -- (2,2);
            \draw (2,2) -- (3,1);
            \draw (3,1) -- (4,2);
            \draw (4,2) -- (5,1);
            \draw (5,1) -- (6,0);
            \draw (6,0) -- (7,1);
            \draw (7,1) -- (8,0);
            \draw (8,0) -- (9,1);
            \draw (9,1) -- (10,2);
            \draw (10,2) -- (11,3);
            \draw (11,3) -- (12,2);
            \draw (12,2) -- (13,1);
            \draw (13,1) -- (14,0);
            \node at (14,-0.5) {$14$};
            \draw[dotted] (0,4) -- (14,4) -- (14,0);
\end{tikzpicture}
       \caption{ {\em
        $\pi = \bigl\{ \, \{1,6\}, \, \{2,3\}, \, \{4,5\}, \, \{7,8\},
         \, \{9,14\}, \, \{10, 13\}, \, \{11,12\} \, \bigr\}$ in $NC_2 (14)$,
         and graph of the associated function
         $f_{\pi} : [0,14] \to [0, \infty )$. } }
        \label{fig:fpi}
    \end{figure}

\vspace{10pt}

In the setting of the present paper, the above mentioned notion of area becomes
a random variable $\Area_n : \NCord_2(2n) \to \bR$, defined by
$\Area_n (\pi,u) := \mathrm{area} ( \pi )$ for $(\pi,u) \in \NCord_2 (2n)$.

\vspace{6pt}

\begin{theorem}   \label{thm:111}
The expectation of the random variable
$\Area_n : \NCord_2(2n) \to \bR$ is 
\begin{equation}   \label{eqn:111a}
E[\Area_n] = (2n+1) \, \sum_{k=1}^n \frac{1}{2k+1},
\ \ n \in \bN \text{.} 
\end{equation}
As a consequence of (\ref{eqn:111a}), it follows that 
$E [ \Area_n ] \sim n \, \ln n$ for $n \to \infty$.
\end{theorem}

\vspace{6pt}

An interesting aspect of the formula $E [ \Area_n ] \sim n \, \ln n$ from  
Theorem \ref{thm:111} is that it gives a clue on how the graphs of the functions 
$f_{\pi}$ ought to be rescaled for asymptotic analysis. Since the rescaling of
the horizontal axis must be, clearly, by a factor of $n$ (in order to shrink the 
domain of $f_{\pi}$ to the fixed interval $[0,2]$), it follows that the rescaling 
of the vertical axis ought to be made by a factor of $\ln n$.  This is in contrast to 
the analogous discussion that one would have for unordered pair-partitions in 
$NC_2 (n)$, where the calculation of expected areas gives values 
$\sim \sqrt{\pi} \cdot n^{3/2}$ (see e.g.~\cite[Theorem 2.1]{MeSpVe1996}),
and the suggested vertical rescaling is thus by $\sqrt{n}$ rather than $\ln n$.

\vspace{10pt}

\subsection{Organization of the paper.}
Besides the Introduction, the paper has 7 sections.

\vspace{6pt}

\noindent
-- Section \ref{section:2} gives the precise description of the $\NCord$-tree, 
and of the two kinds of recursions mentioned in Remark \ref{rem:19}.

\vspace{6pt}

\noindent
-- Sections \ref{section:3} and \ref{section:4} deal with the recursion of the 
first kind for sequences of random variables on the $\NCord (n)$'s.  The main 
point here is that, by relying on the $\NCord$-tree structure, one can convert 
the said kind of recursion into a recursion satisfied by a ``combinatorial version'' 
of the Laplace transforms of the random variables under consideration.  The precise 
description of how this goes is given in Theorem \ref{thm:32}, which is then used in 
Section \ref{section:4} in order to prove the Theorems \ref{thm:16} and \ref{thm:17} 
stated above.

\vspace{6pt}

\noindent
-- In Section \ref{section:5} we make a concise review of what are monotonic cumulants, 
and we show how calculations from Section \ref{subsection:4-1} can be viewed as giving
information on the so-called ``monotonic Poisson process''.

\vspace{6pt}

\noindent
-- In Section \ref{section:6} we return to the main line of the paper, and we pursue 
a development analogous to the one from Sections \ref{section:3} and \ref{section:4},
in connection to the $\NCord$-recursion property of the second kind.  The resulting 
recursion for combinatorial Laplace transforms is stated in Theorem \ref{thm:61},
and its application to the counting of outer blocks is in Corollary \ref{cor:63}.

\vspace{6pt}

\noindent
-- In Section \ref{section:7} we switch to looking at sequences of random 
variables on the $\NCord_2 (2n)$'s.  We examine how the formulas obtained in 
Section \ref{section:6} about the recursion of the second kind can be adjusted 
to the $\NCord_2$-setting, and we use that in order to prove Theorem \ref{thm:110}.

\vspace{6pt}

\noindent
-- In the final Section \ref{section:8} we discuss the ``area'' statistic for 
monotonically ordered non-crossing pair-partitions, and we give the proof 
of Theorem \ref{thm:111}.

\vspace{16pt}

\section{The \texorpdfstring{$\NCord$}{NCord}-tree and some recursions related to it}
\label{section:2}

\setcounter{equation}{0}

\subsection{Description of the \texorpdfstring{\boldmath{$\NCord$}}{NCord}-tree.}
\label{subsection:2-1}

$\ $

\noindent
In this subsection we show how $\NCord := \sqcup_{n=1}^{\infty} \NCord (n)$ can be
organized as set of vertices of a rooted tree, in the way anticipated in 
Remark \ref{rem:13}.  In order to indicate the edges of this tree we will explain, 
in Definition \ref{def:22}, how one assigns a parent $\parent ( \pi , u) \in \NCord (n-1)$ 
to a $( \pi, u ) \in \NCord (n)$ for $n \geq 2$. First, some preliminary notation. 

\vspace{6pt}

\begin{notation}   \label{def:21}
$1^o$ For every $( \pi , u ) \in \sqcup_{n=1}^{\infty} \NCord (n)$, we will use the 
notation ``$J( \pi, u )$'' for the block of $\pi$ to which $u$ assigns its maximal value,
$u \bigl( \, J( \pi, u ) \, \bigr) = | \pi |$.
As noticed in Remark \ref{rem:12}, $J$ is sure to be an {\em interval-block} of $\pi$. 

\vspace{6pt}

\noindent
$2^o$ A {\em restrict-and-relabel} notation: for $\pi \in NC(n)$ and 
$\emptyset \neq W \subseteq \{ 1, \ldots , n \}$, we let $\pi_W$ denote
the partition of $\{ 1, \ldots , |W| \}$ which is obtained by restricting 
$\pi$ to $W$ and then re-denoting the elements of $W$ as $1, \ldots, |W|$ 
in increasing order.  It is immediate that the non-crossing property 
is preserved by this operation, thus $\pi_W \in NC( \, |W| \, )$.
\end{notation}

\vspace{6pt}

\begin{definition}   \label{def:22}
Let $( \pi , u ) \in \NCord (n)$, where $n \geq 2$.  We will construct a 
$( \rho , v ) \in \NCord (n-1)$, which will be called {\em the parent} of 
$( \pi, u )$ and will be denoted as $\parent ( \pi , u )$.
Towards that end, we consider the block $J( \pi, u ) \in \pi$ from Notation \ref{def:21}.1,
and we let $m := \max \bigl( \, J( \pi, u ) \, \bigr) \in \{ 1, \ldots , n \}$.
The ``$\rho$'' part in the desired $( \rho , v )$ is defined as 
\begin{equation}  \label{eqn:22a}
\rho := \pi_{ \{ 1, \ldots , n \} \setminus \{ m \} } 
\in NC(n-1)
\end{equation}
(using the  ``restrict-and-relabel'' Notation \ref{def:21}.2). 
For the ``$v$'' part in $( \rho , v )$ we start from use the monotonic ordering 
$u : \pi \to \{ 1, \ldots , | \pi | \}$, and we go as follows.

\vspace{6pt}

\noindent
-- Case 1: $|J( \pi, u) | > 1$.  In this case we have $| \rho | = | \pi |$, and 
we can write $\rho = \{ \phi (V) \mid V \in \pi \}$, where $\phi$ is
the unique increasing bijection from 
$\{ 1, \ldots , n \} \setminus \{ m \}$ onto $\{ 1, \ldots , n-1 \}$.
We define $v : \rho \to \{ 1, \ldots , | \rho | \}$
by putting $v \bigl( \, \phi (V) \, \bigr) := u(V)$ for every $V \in \pi$.

\vspace{6pt}

\noindent
-- Case 2: $|J( \pi, u )| = 1$, and hence $J( \pi, u ) = \{ m \}$.  
Here the set-difference $\{ 1, \ldots , n \} \setminus \{ m \}$ 
completely eliminates the block $J( \pi , u )$, and thus 
$| \rho | = | \pi | - 1$.  But we still use a procedure similar to the one 
described in Case 1: for the same $\phi$ as considered there, we now write
$\rho$ as  $\bigl\{ \phi (V) \mid V \in \pi, \, V \neq \{ m \}, \,  \bigr\}$,
and we define the ordering $v : \rho \to \{ 1, \ldots , | \rho | \}$
by putting $v \bigl( \, \phi (V) \, \bigr) := u(V)$ for every 
$V \in \pi$, $V \neq \{ m \}$.
\end{definition}

\vspace{6pt}

\begin{example}  \label{example:23}
$1^o$ Say that $n=6$, 
$\pi = \bigl\{ \, \{ 1,2,6 \}, \, \{ 3 \}, \, \{ 4,5 \} \, \bigr\}$,
and we consider the monotonic ordering $u : \pi \to \{ 1,2,3 \}$ defined by
$u \bigl( \, \{ 1,2,6 \} \, \bigr) = 1,
\ u \bigl( \, \{ 3 \} \, \bigr) = 2,
\ u \bigl( \, \{ 4,5 \} \, \bigr) = 3,$
with $J( \pi, u ) = \{4,5 \}$ and $m := \max \bigl( \, J( \pi, u ) \, \bigr) = 5$. 
Then $\parent ( \pi, u ) = ( \rho , v ) \in \NCord (5)$, where
\[
\rho = \pi_{ \{ 1, \ldots , 6 \} \setminus \{ 5 \} }
= \bigl\{ \, \{ 1,2,5 \}, \, \{ 3 \}, \, \{ 4 \} \, \bigr\} \in NC(5),
\]
and the ordering $v$ of the blocks of $\rho$ is:
$v \bigl( \, \{ 1,2,5 \} \, \bigr) = 1,
\ v \bigl( \, \{ 3 \} \, \bigr) = 2,
\ v \bigl( \, \{ 4 \} \, \bigr) = 3.$

\vspace{6pt}

\noindent
$2^o$ Suppose we look at the same $\pi \in NC(6)$ as in $1^o$ above, 
but considered now with the monotonic ordering $u' : \pi \to \{ 1,2,3 \}$
defined by
$u' \bigl( \, \{ 1,2,6 \} \, \bigr) = 1,
\ u' \bigl( \, \{ 4,5 \} \, \bigr) = 2,
\ u' \bigl( \, \{ 3 \} \, \bigr) = 3,$
with $J( \pi, u' ) = \{ 3 \}$ and 
$m' := \max \bigl( \, J( \pi, u' ) \, \bigr) = 3$. 
Then $\parent ( \pi, u' ) = ( \rho' , v' ) \in \NCord (5)$, where
\[
\rho ' = \pi_{ \{ 1, \ldots , 6 \} \setminus \{ 3 \} }
= \bigl\{ \, \{ 1,2,5 \}, \, \{ 3, 4 \} \, \bigr\} \in NC(5),
\]
and the ordering $v'$ of the blocks of $\rho '$ is:
$v' \bigl( \, \{ 1,2,5 \} \, \bigr) = 1,
\ v' \bigl( \, \{ 3,4 \} \, \bigr) = 2.$
\end{example}

\vspace{6pt}

\begin{notation}   \label{def:24}
For $n \geq 2$ in $\bN$ and $( \rho, v ) \in \NCord (n-1)$ we will denote
\begin{equation}   \label{eqn:24a}
C( \rho, v ) := \{ ( \pi, u ) \mid \parent ( \pi, u ) = ( \rho, v ) \} 
\subseteq \NCord (n);
\end{equation}
that is, $C( \rho , v )$ is the {\em set of children} of $( \rho , v )$ 
in the $\NCord$-tree.
\end{notation}

\vspace{6pt}

\begin{remark}   \label{rem:25}
{\em (Homogeneity property of the $\NCord$-tree.)}
Let $n \geq 2$ be in $\bN$ and let $( \rho , v ) \in \NCord (n-1)$.
Then $| C( \rho , v ) | = n+1$.  More precisely, one can write
\begin{equation}  \label{eqn:25a}
C( \rho, v ) = \bigl\{ ( \pi_1, u_1 ), \ldots , ( \pi_n , u_n), 
( \pi_{n+1} , u_{n+1} ) \bigr\},
\end{equation}
where:

\noindent
-- For every $1 \leq m \leq n$, the partition $\pi_m \in NC(n)$ is obtained by 
inserting a singleton block in the picture of $\rho$, with the added singleton
placed on position $m$, and where this singleton block $\{ m \}$ gets 
to have maximal label in the monotonic ordering $u_m$ of the blocks of $\pi_m$
-- that is, one has $J( \pi_m, u_m ) = \{ m \}$.

\vspace{6pt}

\noindent
-- The partition $\pi_{n+1}$ is obtained by considering the interval-block 
$J ( \rho, v )  = \{ p, \ldots , q \}$ with maximal $v$-label, 
and by ``elongating'' $\{ p, \ldots , q \}$ to become the block with maximal label 
in the monotonic ordering $u_{n+1}$ of $\pi_{n+1}$ -- that is, one has
\begin{equation}   \label{eqn:26b}
J( \pi_{n+1} , u_{n+1} ) = \{ p, \ldots , q, q+1 \} \in \pi_{n+1}.
\end{equation}
\end{remark}

\vspace{10pt}

\subsection{The \texorpdfstring{\boldmath{$\NCord$}}{NCord}-recursion of the first kind.}
\label{subsection:2-2}

\begin{definition}   \label{def:26}
Let $( Z_n : \NCord (n) \to \bR )_{n=1}^{\infty}$ be real random variables, let 
$k \in \bN$, and let $\ur = (\pp_1, \ldots , \pp_k) \in \bR^k$.  We will say 
that the sequence of $Z_n$'s is {\em recursive of the first kind}
with input $\ur$ to mean that the following condition is fulfilled:
\begin{equation}   \label{eqn:26a}
\begin{array}{lll}
\vline & \mbox{For every $n \geq 2$ and $(\pi, u ) \in \NCord (n)$,} & \vline \\
\vline & \mbox{with $\parent ( \pi , u) = ( \rho, v ) \in \NCord (n-1)$, one has:} & \vline  \\
\vline &                &  \vline   \\
\vline &  \begin{array}{cc}
         \mbox{$\ $} \ Z_n ( \pi, u ) = Z_{n-1} ( \rho,v ) + \pp_1, 
                               & \mbox{ if $| J( \pi, u ) | = 1$,}         \\
         \mbox{$\ $} \ \cdots \cdots \cdots \cdots \cdots \cdots
                               & \cdots \cdots \cdots                      \\
         \mbox{$\ $} \ Z_n ( \pi, u) = Z_{n-1} ( \rho,v ) + \pp_k, 
                               & \mbox{ if $| J( \pi, u ) | = k$,}         \\
         \mbox{$\ $}           &                                           \\
         \mbox{and } Z_n ( \pi, u) = Z_{n-1} ( \rho,v ), 
                               & \mbox{ if $| J( \pi, u ) | > k$.} 
         \end{array}  & \vline
\end{array}
\end{equation}
\end{definition}

\vspace{6pt}

\begin{remark}  \label{rem:27}
$1^o$ The way to think about the recursive condition stated in (\ref{eqn:26a})
is that we generally want to have 
$Z_n ( \pi, u) = Z_{n-1} \bigl( \, \parent ( \pi , u ) \, \bigr)$,
but we allow exceptions when the special block $J( \pi,u) \in \pi$ 
has small size -- namely, if $| J( \pi , u ) | = j \leq k$ then
$Z_n ( \pi, u) = Z_{n-1} \bigl( \, \parent ( \pi , u ) \, \bigr) + \pp_j$.

\vspace{6pt}

\noindent
$2^o$ In the setting of Definition \ref{def:26}, it is immediate 
that if the input tuple $\ur$ is given and if $Z_1$ is given, then all the 
$Z_n$'s are uniquely determined. However,
the way we would like to use this definition goes, in some sense, in reverse:
given a sequence of $Z_n$'s which we deem to be of interest, we want to see if 
it wouldn't be possible to fit a $k$ and a 
$\ur = (\pp_1, \ldots , \pp_k)$ such that the $Z_n$'s are recursive of the 
first kind with input $\ur$.
This turns out to work for a number of natural examples of block-counting random 
variables, as shown below.

Note that in the above mentioned process of fitting an input vector 
$\ur = (\pp_1, \ldots , \pp_k)$, there is some freedom in the choice of $k$: 
we can always replace $k$ with an $\ell > k$, and replace the input tuple 
$\ur \in \bR^k$ with  $(\pp_1, \ldots , \pp_k, 0, \ldots 0) \in \bR^{\ell}$. 
\end{remark}

\vspace{6pt}

\begin{example}   \label{example:28}
The block-counting random variables $Y_n$ introduced in Notation \ref{rem:15} 
satisfy the condition (\ref{eqn:26a}), with $k =1$ and $\ur = (1) \in \bR^1$. 
This amounts to the fact that, for every $n \geq 2$ and $( \pi, u ) \in \NCord (n)$ 
with parent $\parent ( \pi , u ) =: ( \rho , v ) \in \NCord (n-1)$, one has:
\begin{equation}    \label{eqn:28a}
| \pi | = \left\{  \begin{array}{ll}
| \rho | + 1, & \mbox{if $|J( \pi, u )| = 1$,}   \\
| \rho |,     & \mbox{if $|J( \pi, u )| > 1$.}   
\end{array}   \right. 
\end{equation}
For the verification of (\ref{eqn:28a}), we check the two possible cases: if 
$|J( \pi, u )| = 1$ then $\pi$ was obtained by appending a singleton block to $\rho$, 
and thus $| \pi | = | \rho | + 1$; while if $|J( \pi, u )| \geq 2$ then $\pi$ was 
obtained by elongating the block $J( \rho , v )$ of $\rho$, and thus $| \pi | = | \rho | $. 
\end{example}

\vspace{6pt}

\begin{example}  \label{example:29}
Let us fix an $\ell \in \bN$ and consider the random variables 
$\bigl( \, Y_n^{( \ell )} \, \bigr)_{n=1}^{\infty}$ which were introduced in 
Equation (\ref{eqn:17a}) of Theorem \ref{thm:17}, counting blocks of cardinality $\ell$.  
We claim that the $Y_n^{ ( \ell ) }$'s satisfy the condition (\ref{eqn:26a}), with 
$k = \ell + 1$ and $\ur = (0, \ldots , 0, 1, -1) \in \bR^{\ell + 1}$.  This amounts
to the fact that for $n \geq 2$ and $( \pi, u ) \in \NCord (n)$ 
with parent $\parent ( \pi , u ) =: ( \rho , v ) \in \NCord (n-1)$, one has:
\begin{equation}    \label{eqn:29a}
Y_n^{( \ell )} ( \pi, u ) = \left\{  \begin{array}{ll}
Y_n^{ ( \ell ) } ( \rho , v ) + 1, & \mbox{if $|J( \pi, u )| = \ell$,}       \\
Y_n^{ ( \ell ) } ( \rho , v ) - 1, & \mbox{if $|J( \pi, u )| = \ell + 1$,}   \\
Y_n^{ ( \ell ) } ( \rho , v ),     & \mbox{otherwise.}   
\end{array}   \right. 
\end{equation}
For the verification of (\ref{eqn:29a}), we check the three possible cases listed there.  
We will do the verification under the assumption that $\ell \geq 2$ (and leave it as 
exercise to the reader to supply the slight adjustments needed for $\ell = 1$). 

\vspace{6pt}

\noindent
{\em Case 1:} $|J( \pi, u )| = \ell$. Since $\ell \geq 2$, it must be that $J( \pi, u )$ 
was obtained by elongating the block $J( \rho , v )$ of $\rho$, which had 
$| J( \rho, v ) | = \ell -1$.  So $\pi$ indeed has one more block of cardinality $\ell$ 
than $\rho$ did -- this is precisely the block $J ( \pi, u )$. 

\vspace{6pt}

\noindent
{\em Case 2:} $|J( \pi, u )| = \ell + 1$.  This is similar to Case 1, but 
where we now have that $| J( \rho , v ) | = \ell \neq| J ( \pi , u ) | \neq \ell$; 
hence in this case $\pi$ has one less block of cardinality $\ell$ than $\rho$ did.

\vspace{6pt}

\noindent
{\em Case 3:} $|J( \pi, u )| \not\in \{ \ell,  \ell + 1 \}$.  In this case, there are no blocks 
of cardinality $\ell$ that would either appear or disappear in the process of passing 
from $( \rho , v )$ to $( \pi, u )$, hence $\pi$ has the same number of blocks of cardinality 
$\ell$ as $\rho$ (as claimed in (\ref{eqn:29a})).
\end{example}

\vspace{10pt}

\begin{example}   \label{example:210}
Consider the random variables 
$\bigl( \, Y_n^{( \geq 3 )} \, \bigr)_{n=1}^{\infty}$ which were introduced in 
Equation (\ref{eqn:18a}) of Remark \ref{rem:18}, counting all the blocks which do not have 
cardinality $1$ or $2$.  The $Y_n^{ ( \geq 3 ) }$'s satisfy the condition (\ref{eqn:26a}), 
with $k = 3$ and $\ur = (0,0,1) \in \bR^3$.  The verification of this claim is similar to 
the one shown in the preceding example, and is left to the reader.
\end{example}

\vspace{10pt}

\subsection{The \texorpdfstring{\boldmath{$\NCord$}}{NCord}-recursion of the second kind.}
\label{subsection:2-3}

\begin{definition}  \label{def:211}
Let $\alpha,\beta, q \in \bZ$ be given.  We will say that a sequence of 
random variables $(Z_n : \NCord (n) \to \bZ)_{n=1}^{\infty}$ is 
{\em recursive of the second kind} with input $(\alpha,\beta; q)$ when it has
the following property: for every $n \geq 2$ and $(\rho,v) \in \NCord (n-1)$, 
the set $C( \rho, v )$ of children of $( \rho , v )$ 
has a subset $C_o ( \rho, v )$ such that:
\begin{equation}   \label{eqn:211a}
\left\{   \begin{array}{ll}
\mathrm{(1)} & \mbox{If $( \pi, u ) \in C_o ( \rho,v )$, then
      $Z_n(\pi,u) = Z_{n-1}(\rho,v) + \alpha$;}    \\ 
\mathrm{(2)} & \mbox{if $( \pi, u ) \in C( \rho, v ) \setminus C_o ( \rho, v)$, then
      $Z_n(\pi,u) = Z_{n-1}(\rho,v) + \beta$;}    \\ 
\mbox{$\ $} \mathrm{ and} &                                        \\
\mathrm{(3)} & \mbox{$| C_o( \rho, v) | = Z_{n-1}(\rho,v) + q$.}
\end{array}   \right.
\end{equation}
\end{definition}

\vspace{10pt}

\begin{example}   \label{example:212}
For every $n \in \bN$, let $\OuT_n : \NCord(n) \to \bN$ be defined by
\begin{equation}  \label{eqn:212a}
\OuT_n ( \pi, u ) := \ \vline \, \{ V \in \pi \mid \, V \mbox{ is outer} \} \ \vline \ ,
\ \ ( \pi, u ) \in \NCord (n),
\end{equation}
where recall that a block $V \in \pi$ is said to be \emph{outer} 
when there is no $W \in \pi$ with $\min (W) < \min (V)$ and $\max (W) > \max (V)$.
We will verify that the sequence of random variables $( \OuT_n )_{n=1}^{\infty}$ 
is recursive of the second kind, with parameters $\alpha=1$, $\beta = 0$ and $q=1$.
Towards that end, we fix an $n \geq 2$ and a $( \rho , v ) \in \NCord (n-1)$,
and we consider the set $C( \rho , v )$ of children of $( \rho, v )$ in the 
$\NCord$-tree.  We want to find a a subset $C_o ( \rho, v ) \subseteq C( \rho, v )$ 
such that the requirements from (\ref{eqn:211a}) are fulfilled, i.e.~such that:
\begin{equation}   \label{eqn:212b}
\left\{   \begin{array}{ll}
\mathrm{(1)} & \mbox{If $( \pi, u ) \in C_o ( \rho,v )$, then
      $\OuT_n ( \pi , u ) = \OuT_{n-1} (\rho,v) + 1$;}    \\ 
\mathrm{(2)} & \mbox{if $( \pi, u ) \in C( \rho, v ) \setminus C_o ( \rho, v)$, then
      $\OuT_n(\pi,u) = \OuT_{n-1}(\rho,v)$;}    \\ 
\mbox{$\ $} \mathrm{ and} &                                        \\
\mathrm{(3)} & \mbox{$| C_o( \rho, v) | = \OuT_{n-1} (\rho,v) + 1$.}
\end{array}   \right.
\end{equation}
In order to find $C_o ( \rho, v )$, we resort to the explicit description of 
$C( \rho, v )$, as given in Equation (\ref{eqn:25a}) of Remark \ref{rem:25}: 
\[
C( \rho, v ) = \{ (\pi_1, u_1), \ldots , ( \pi_n, u_n), ( \pi_{n+1}, u_{n+1} ) \},
\]
where for $1 \leq m \leq n$ we have $J( \pi_m, u_m) = \{m\}$, while for $m = n+1$ we 
get the block $J( \pi_{n+1}, u_{n+1} )$ of $\pi_{n+1}$ by elongating (that is, by
appending one more element to) the interval-block $J( \rho, v )$ of $\rho$.

Now, by using the non-crossing property, it is easy to see that the outer blocks of $\rho$
can be organized in a list $W_1, \ldots , W_s$ such that
\[
\min (W_1) = 1, \ \min (W_2) = \max (W_1) + 1, \ldots ,
\min (W_s) = \max (W_{s-1} ) + 1, \ \max (W_s) = n-1.
\]
(These blocks form, in some sense, an ``outer frame'' for $\rho$.) 
Referring to the explicit description of $C( \rho , v )$ reviewed above, 
we observe that for $1 \leq m \leq n+1$ we have:
\[
\OuT_n ( \pi_m , u_m ) = \left\{ \begin{array}{ll}
\OuT_{n-1} ( \rho, v ) + 1, & \mbox{ if } m \in \{ \min (W_1), \ldots , \min (W_s), n \}, \\
\OuT_{n-1} ( \rho, v),      & \mbox{otherwise.}
\end{array}  \right.
\]
This is because, when passing from $\rho$ to a $\pi_m$, the outer blocks of
$\rho$ always become outer blocks of $\pi_m$, and there is one possible additional outer
block of $\pi_m$ showing up -- this occurs when $1 \leq m \leq n$ and the singleton block added 
to $\rho$ in order to obtain $\pi_m$ is placed ``in between'' two consecutive outer blocks 
$W_{i-1}$ and $W_i$ of $\rho$ (also allowing here the possibility that the new singleton 
block is placed to the left of $W_1$ or to the right of $W_s$).

So then, the subset $C_o ( \rho, v) \subseteq C ( \rho , v )$ defined as 
\[
C_o ( \rho , v ) 
:= \Bigl\{ \pi_m \mid m \in \{ \min (W_1), \ldots , \min (W_s), n \} \, \Bigr\}
\]
satisfies the requirements (1) and (2) in (\ref{eqn:212b}).  Finally, we observe that 
this choice of $C_o ( \rho, v )$ also fulfills the requirement (3) in (\ref{eqn:212b}),
since $| C_o ( \rho, v ) | = s+1 = \OuT_{n-1} ( \rho, v ) +1$.
\end{example}

\vspace{10pt}

\section{Laplace transforms for the
\texorpdfstring{$\NCord$}{NCord}-recursion of the first kind}
\label{section:3}

\setcounter{equation}{0}

\noindent
Let $n \in \bN$ and let $Z_n : \NCord (n) \to \bR$ be a random variable, considered 
with respect to uniform distribution on $\NCord (n)$. 
The Laplace transform of $Z_n$ is the complex function 
\[
\lambda \mapsto E[ e^{\lambda Z_n} ] = \frac{1}{ | \NCord (n) | }
\sum_{ ( \pi,u) \in \NCord (n) } e^{\lambda Z_n ( \pi, u ) },
\ \ \lambda \in \bC.
\]
For the purpose of this paper it is convenient that in the latter equation 
we restrict $\lambda$ to $\bR$, put $t = e^{\lambda} \in (0, \infty )$, 
and set aside the normalization constant $| \NCord (n) |$.  This leads 
to the following version of the Laplace transform, which is defined on $(0, \infty )$
and, in our examples of interest where $Z_n$ takes 
values in $\bN \cup \{ 0 \}$, is a polynomial function.

\vspace{6pt}

\begin{definition-and-remark}  \label{def:31}
Let $n \in \bN$ and let $Z_n : \NCord (n) \to \bR$ be a random variable, considered 
with respect to uniform distribution on $\NCord (n)$.  We will use the term 
{\em combinatorial Laplace transform of $Z_n$} for the function 
$L_n : (0, \infty ) \to \bR$ defined by
\begin{equation}   \label{eqn:31a}
L_n (t) = \sum_{ ( \pi, u ) \in \NCord (n) } t^{Z_n ( \pi, u )},
\ \ t > 0.
\end{equation}

We note that $L_n$ contains significant information about the random variable $Z_n$; 
in particular, direct calculation shows that the expectation and variance of $Z_n$ are 
retrieved as 
\begin{equation}   \label{eqn:31b}
E[ Z_n ] = \frac{L_n ' (1)}{L_n (1)} \ \mbox{ and }
\mathrm{Var} [Z_n] = E[ Z_n^2 ] - \bigl( \, E[Z_n] \, \bigr)^2 
= \frac{L_n '' (1) + L_n ' (1)}{L_n (1)} - \Bigl( \, \frac{L_n ' (1)}{L_n (1)} \, \Bigr)^2,
\end{equation}
where the $L_n (1)$ in the denominators keeps track of normalizations,
$L_n (1) = | \NCord (n) | = (n+1)!/2$.
\end{definition-and-remark}

\vspace{6pt}

The next theorem says that if a sequence of random variables 
$( Z_n : \NCord (n) \to \bR )_{n=1}^{\infty}$ is recursive of the first kind (in the sense 
of Definition \ref{def:26}), then the combinatorial
Laplace transforms $L( Z_n )$ obey a tractable recursion formula.  

\vspace{6pt}

\begin{theorem}   \label{thm:32}
Let $( Z_n : \NCord (n) \to \bR )_{n=1}^{\infty}$ be a sequence of real random 
variables which is recursive of the first kind with input
$\ur = (\pp_1, \ldots , \pp_k) \in \bR^k$, where $k \geq 2$. 
For every $n \in \bN$, let $L_n : ( 0, \infty ) \to \bR$ be defined as in 
Equation (\ref{eqn:31a}).  Then one has: 
\begin{equation}   \label{eqn:32a}
L_n (t) = (1 + nt^{\pp_1}) L_{n-1} (t) + \sum_{j=2}^k (n-j+1) \, (t^{\pp_j} - 1)
\, t^{\pp_1 + \cdots + \pp_{j-1}} \, L_{n-j} (t),
\end{equation}
holding for $n > k$ in $\bN$ and $t \in ( 0, \infty )$.
\end{theorem}

\vspace{10pt}

Towards the proof of Theorem \ref{thm:32}, we first prove two lemmas.

\vspace{6pt}

\begin{lemma}  \label{lemma:33}
Consider the framework and notation of Theorem \ref{thm:32}. Let 
$\ell \in \{ 2, \ldots , k \}$, let $m$ be an integer such that $m \geq \ell$,
and let $\parent^{( \ell -1 )} : \NCord (m) \to \NCord (m - \ell + 1)$ be the 
map obtained by iterating $\ell - 1$ times the parent map $\parent$ of the 
$\NCord$-tree.  

\vspace{6pt}

\noindent
$1^o$ Let $( \rho , v ) \in \NCord (m)$ be such that 
$|J( \rho , v)| = \ell$, and let $( \sigma , w ) 
:= \parent^{ ( \ell -1 ) } ( \rho, v ) \in  \NCord (m - \ell + 1)$.
One has that
\begin{equation}  \label{eqn:33a}
Z_m ( \rho , v ) = Z_{m- \ell + 1} ( \sigma , w) + ( \pp_2 + \cdots + \pp_{\ell} ).
\end{equation}

\vspace{6pt}

\noindent
$2^o$ The map $\parent^{ ( \ell -1 ) }$ establishes a bijection between 
$\{ ( \rho , v ) \in \NCord (m) \mid \ |J( \rho , v)| = \ell \}$ and
$\{ ( \sigma , w ) \in \NCord (m - \ell + 1) \mid \ |J( \sigma , w)| = 1 \}$.
\end{lemma}

\begin{proof} $1^o$ Let us write explicitly how the $\ell - 1$ iterations
of the parent map apply to $( \rho , v )$, in order to reach the partition 
$( \sigma , w )$.  We thus denote:
\[
\begin{array}{ll}
\vline  &  ( \rho_0, v_0 ) := ( \rho , v ),
            \ \ ( \rho_1 , v_1) := \parent ( \rho_0, v_0 ) \in \NCord (m-1),   \\
\vline  &                                                                     \\
\vline  & ( \rho_2 , v_2) := \parent ( \rho_1, v_1 ) \in \NCord (m-2), 
          \ldots , ( \rho_{\ell -1} , v_{\ell -1} ) 
          := \parent ( \rho_{\ell -2}, v_{\ell -2} ),                          \\
\vline  &                                                                     \\
\vline  & \mbox{$\ $ to the effect that } 
          ( \sigma, w) = ( \rho_{\ell - 1}, v_{\ell -1} )\in \NCord (m- \ell + 1).
\end{array}
\]
From the definition of the parent map we infer that for every 
$1 \leq i \leq \ell - 1$ we have
$J( \rho_i, v_i ) \subseteq J( \rho_{i-1}, v_{i-1} ),
\mbox{ with }
|J( \rho_i, v_i )| \, = \, |J( \rho_{i-1}, v_{i-1} )| - 1;$
this implies that
\begin{equation}   \label{eqn:33b}
|J( \rho_i, v_i ) | = \ell -i, \ \ \forall \, 0 \leq i \leq \ell - 1.
\end{equation}
The latter information about the cardinalities $| J( \rho_i, v_i ) |$
can be plugged into the recursive condition (\ref{eqn:26a}) satisfied by the 
$Z_n$'s, in order to get that: 
\begin{equation}   \label{eqn:33c}
\begin{array}{llcc}
\vline & Z_m (\rho_0 , v_0 ) & = & Z_{m-1} ( \rho_1 , v_1 ) + \pp_{\ell},       \\ 
\vline & Z_{m-1} ( \rho_1 , v_1 ) & = & Z_{m-2} ( \rho_2 , v_2 ) + \pp_{\ell -1}, \\ 
\vline & \cdots \cdots \mbox{$\ $} \cdots \cdots \mbox{$\ $} \cdots \cdots 
             & \cdots & \cdots \cdots \mbox{$\ $} \cdots \cdots  \mbox{$\ $} \cdots \cdots \\
\vline & Z_{m-\ell+2} ( \rho_{\ell -2} , v_{\ell -2} )
         & = & Z_{m-\ell + 1} ( \rho_{\ell -1} , v_{\ell -1} ) + \pp_2.
\end{array}
\end{equation}
Adding together all the equations listed in (\ref{eqn:33c}) then leads to 
\[
Z_m (\rho_0 , v_0 ) = 
Z_{m-\ell + 1} ( \rho_{\ell -1} , v_{\ell -1} ) + ( \pp_2 + \cdots + \pp_{\ell} ),
\]
which is precisely the required formula (\ref{eqn:33a}).

\vspace{6pt}

\noindent
$2^o$ Let $( \rho , v ) \in \NCord (m)$ be such that $|J( \rho , v)| = \ell$, and
let $( \sigma , w ) := \parent^{ ( \ell - 1) } ( \rho , v ) \in \NCord (m - \ell + 1)$.
Upon reviewing the argument that proved $1^o$ above, we see that
the Equation (\ref{eqn:33b}) found there gives in particular
$|J( \sigma , w)| \, = \, |J( \rho_{\ell -1}, v_{\ell -1} ) | 
= \ell - ( \ell -1 ) = 1.$
Thus the map $\parent^{ ( \ell - 1 ) }$ does indeed send the set 
$\{ ( \rho , v ) \in \NCord (m) \mid \ |J( \rho , v)| = \ell \}$ into
$\{ ( \sigma , w ) \in \NCord (m - \ell + 1) \mid \ |J( \sigma , w)| = 1 \}$.

On the other hand, let us consider the map
\[
\begin{array}{rcl}
Q^{ ( \ell - 1 ) } 
: \{ ( \sigma , w ) \in \NCord (m - \ell + 1) \mid \ |J( \sigma , w)| = 1 \}
& \longrightarrow & \NCord (m),   \\
( \sigma , w ) & \mapsto & ( \rho , v ),
\end{array}
\]
where $(\rho , v)$ is obtained by elongating the singleton block $J( \sigma , w )$
to an interval block of cardinality $\ell$, and by keeping the ordering of blocks 
(thus the elongated block of size $\ell$ is the one which becomes $J( \rho , v )$).
The reader should have no difficulty 
in checking that the map $Q^{ ( \ell -1 ) }$ so defined is an inverse for the 
map from $\{ ( \rho , v ) \in \NCord (m) \mid \ |J( \rho , v)| = \ell \}$ to
$\{ ( \sigma , w ) \in \NCord (m - \ell + 1) \mid \ |J( \sigma , w)| = 1 \}$ which
is induced by $\parent^{ ( \ell - 1 ) }$.
\end{proof}

\vspace{10pt}

\begin{lemma}  \label{lemma:34}
Consider the framework and notation of Theorem \ref{thm:32}.  For every integer 
$m \geq 2$ and $t \in ( 0, \infty)$, one has that
\begin{equation}   \label{eqn:34a}
\sum_{ \substack{ ( \rho , v ) \in \NCord (m) \\ \mathrm{with} \ |J(\rho, v)| = 1 } }
t^{ Z_m ( \rho,v) } = m \, t^{\pp_1} \, L_{m-1} (t).
\end{equation}
\end{lemma}

\begin{proof} On the left-hand side of (\ref{eqn:34a})
we group the terms according to what is the parent of $( \rho, v )$
in the $\NCord$-tree, and this converts the quantity indicated there into:
\begin{equation}   \label{eqn:34b}
\sum_{( \sigma , w) \in \NCord (m-1)} \Bigl[ 
\, \sum_{ \substack{ ( \rho , v ) \in \NCord (m) \\
          \mathrm{with} \ \parent ( \rho, v ) = ( \sigma , w) \\
          \mathrm{and \ with} \ |J( \rho , v)| = 1} } 
t^{ Z_m ( \rho , v )} \, \Bigr] .
\end{equation}

Let us now focus on an ordered partition 
$( \sigma , w ) \in \NCord (m-1)$, for which we examine the inner sum 
corresponding to $( \sigma , w )$ in (\ref{eqn:34b}).  In view of the 
explicit description of the children of $( \sigma, w )$ in the $\NCord$-tree
observed in Remark \ref{rem:25}, we can list
\[
\bigl\{ ( \rho , v ) \in \NCord (m) \mid \parent ( \rho, v ) = ( \sigma , w)
\ \mathrm{and} \ |J( \rho , v)| = 1 \bigr\} = 
\{ (\rho_1, v_1), \ldots , (\rho_m, v_m) \},
\]
where, for every $1 \leq i \leq m$, 
$( \rho_i , v_i)$ is the unique ordered partition in $\NCord (m)$ such that 
$\parent ( \rho_i , v_i ) = ( \sigma , w )$ and $J( \rho_i, v_i ) = \{ i \}$. 
Since the recursive condition (\ref{eqn:26a}) satisfied by the 
$Z_n$'s assures us that 
$Z_m ( \rho_i , v_i ) = Z_{m-1} ( \sigma , w ) + \pp_1$,
$1 \leq i \leq m$ we infer that
\begin{equation}   \label{eqn:34c}
\sum_{ \substack{ ( \rho , v ) \in \NCord (m) \\
          \mathrm{with} \ \parent ( \rho, v ) = ( \sigma , w) \\
          \mathrm{and \ with} \ |J( \rho , v)| = 1} } 
t^{ Z_m ( \rho , v )}  
= \sum_{i=1}^m t^{ Z_m ( \rho_i, v_i ) } 
= m \, t^{ Z_{m-1} ( \sigma , w ) + \pp_1 }.
\end{equation}

Equation (\ref{eqn:34c}) holds for every $( \sigma, w ) \in \NCord (m-1)$.
This allows us to plug the right-hand side of (\ref{eqn:34c}) as inner sum in 
(\ref{eqn:34b}), in order to conclude that 
\[
\left(  \begin{array}{c} 
\mbox{quantity}  \\
\mbox{from (\ref{eqn:34b})} 
\end{array} \right)
= \sum_{( \sigma , w) \in \NCord (m-1)} \Bigl[ 
\, m \, t^{\pp_1} \, t^{ Z_{m-1} ( \sigma , w )} \, \Bigr]
= m \, t^{\pp_1} \, L_{m-1} (t),
\]
as claimed in (\ref{eqn:34a}).
\end{proof}

\vspace{10pt}

\begin{ad-hoc-item}  \label{proof:35}
{\bf Proof of Theorem \ref{thm:32}.} 
We fix for the proof an $n > k$ in $\bN$ and a $t \in (0, \infty)$ for which 
we will verify that (\ref{eqn:32a}) holds. 
In the sum over $( \pi, u ) \in \NCord (n)$ which defines $L_n (t)$, we 
organize the terms according to what is the parent of $( \pi, u )$ in the 
$\NCord$-tree:
\begin{equation}  \label{eqn:35a}
L_n (t) = \sum_{( \rho , v) \in \NCord (n-1)} \Bigl[ 
\, \sum_{ \substack{ ( \pi , u ) \in \NCord (n) \\
\mathrm{with} \ \parent ( \pi, u ) = ( \rho , v)} } t^{Z_n ( \pi , u )} \, \Bigr] .
\end{equation}

Let us fix for the moment an ordered partition $( \rho , v ) \in \NCord (n-1)$, for which 
we examine the inner sum corresponding to $( \rho , v )$ in (\ref{eqn:35a}).
This sum is indexed by the set 
$C( \rho , v ) = \{ ( \pi_1, u_1 ), \ldots , ( \pi_{n+1}, u_{n+1} ) \}$
of children of $( \rho , v )$ in the $\NCord$-tree, 
with the $( \pi_m , u_m )$'s explicitly described in the paragraph following to 
Equation (\ref{eqn:25a}) of Remark \ref{rem:25}.  By invoking the recursive 
property (\ref{eqn:26a}) in connection to the explicit description of $C( \rho, v )$ 
reviewed above, we find that for $1 \leq m \leq n$ we have
$Z_n ( \pi_m, u_m ) = Z_{n-1} ( \rho , v ) + \pp_1$
(since $( \pi_m, u_m)$ is obtained by inserting a singleton block in $(\rho, v )$).
The remaining value $Z_n ( \pi_{n+1}, u_{n+1} )$ relates to $Z_{n-1} ( \rho , v )$ in 
a way which depends on what is the cardinality of the block 
$J( \pi_{n+1}, u_{n+1} ) \in \pi_{n+1}$.
The latter cardinality is just $1 + |J( \rho , v )|$
(since $( \pi_{n+1}, u_{n+1})$ is obtained by elongating the block $J(\rho, v ) \in \rho$
in order to make it become $J( \pi_{n+1}, u_{n+1} )$). The recursive property (\ref{eqn:26a}) 
thus gives the following possibilities $(*)$ and $(**)$:
\[
\left\{
\begin{array}{ll}
(*)   & \mbox{if } 1 + |J( \rho , v )| > k, \mbox{ then }
Z_n ( \pi_{n+1}, u_{n+1} ) = Z_{n-1} ( \rho, v );   \\
      &                                             \\
(**)  & \mbox{if } 1 + |J( \rho , v )| = j \in \{ 2, \ldots , k \},
\mbox{ then } Z_n ( \pi_{n+1}, u_{n+1} ) = Z_{n-1} ( \rho, v ) + \pp_j.
\end{array}  \right.
\]
Returning to the evaluation of the inner sum which corresponds to 
$( \rho , v )$ in (\ref{eqn:35a}), we see that this sum is
\begin{align*}
\sum_{m=1}^{n+1} t^{Z_n ( \pi_m , u_m )}  
& = n \, t^{ Z_{n-1} ( \rho , v ) + \pp_1 } + t^{ Z_n ( \pi_{n+1}, u_{n+1} ) }  \\
& = n \, t^{ Z_{n-1} ( \rho , v ) + \pp_1 } + t^{ Z_{n-1} ( \rho , v ) } 
\end{align*}
where the correction term appears in connection to the possibility $(**)$, when it is 
equal to $t^{ Z_{n-1} ( \rho , v ) + \pp_j } - t^{ Z_{n-1} ( \rho , v ) }$ (for 
$j = 1 + |J( \rho , v )|$).

\vspace{6pt}

We now unfix the $( \rho , v )$ that was considered throughout the 
preceding paragraph, and plug the result of that discussion into the sum over 
$(\rho , v ) \in \NCord (n-1)$ appearing in (\ref{eqn:35a}).  Then Equation (\ref{eqn:35a})
takes the following form:
\begin{equation}  \label{eqn:35b}
L_n (t) = \sum_{( \rho , v) \in \NCord (n-1)} \Bigl[ 
\, n \, t^{ Z_{n-1} ( \rho , v ) + \pp_1 } + t^{ Z_{n-1} ( \rho , v ) } \, \Bigr]
\end{equation}
\[
+ \sum_{j=2}^k \, \Bigl[ \, \sum_{ \substack{ ( \rho , v ) \in \NCord (n-1) \\
                                \mathrm{with} \ |J( \rho, v )| = j-1} } 
t^{ Z_{n-1} ( \rho , v ) + \pp_j } - t^{ Z_{n-1} ( \rho , v ) } \, \Bigr].
\]
From here we continue as follows.

\vspace{6pt}  

\noindent
{\em Claim 1.} The first sum listed in (\ref{eqn:35b}) is equal to
$(n \, t^{\pp_1} + 1 ) \, L_{n-1} (t).$

\vspace{6pt}

\noindent
{\em Verification of Claim 1.}  The sum in question is 
$\sum_{ ( \rho , v ) \in \NCord (n-1) } 
(n \, t^{\pp_1} + 1 ) t^{ Z_{n-1} ( \rho , v )},$
and, in view of how $L_{n-1}$ is defined, this is indeed 
$(n \, t^{\pp_1} + 1 ) \, L_{n-1} (t)$.

\vspace{6pt}

\noindent
{\em Claim 2.} The second sum listed in (\ref{eqn:35b}) is equal to
\[
\sum_{j=2}^k \, (n-j+1) \, ( t^{\pp_j} - 1) 
\, t^{\pp_1 + \cdots + \pp_{j-1}} \, L_{n-j} (t).
\]

\noindent
{\em Verification of Claim 2.} 
The sum under consideration can be written as: 
\[
\sum_{j=2}^k ( t^{\pp_j} - 1) 
\, \Bigl[ \, \sum_{ \substack{ ( \rho , v ) \in \NCord (n-1) \\
                                \mathrm{with} \ |J( \rho, v )| = j-1} } 
t^{ Z_{n-1} ( \rho , v ) } \, \Bigr],
\]
hence it suffices to verify that we have equalities
\begin{equation}   \label{eqn:35c}
\sum_{ \substack{ ( \rho , v ) \in \NCord (n-1) \\
                                \mathrm{with} \ |J( \rho, v )| = j-1} } 
t^{ Z_{n-1} ( \rho , v ) }
= (n-j+1) \, t^{\pp_1 + \cdots + \pp_{j-1}} \, \, L_{n-j} (t),
\end{equation}
holding for every $j \in \{ 2, \ldots , k \}$.  The case $j=2$ of (\ref{eqn:35c})
says that
\[
\sum_{ \substack{ ( \rho , v ) \in \NCord (n-1) \\
                                \mathrm{with} \ |J( \rho, v )| = 1} } 
t^{ Z_{n-1} ( \rho , v ) }
= (n-1) \, t^{\pp_1} \, \, L_{n-2} (t),
\]
which is just Lemma \ref{lemma:34}, used for $m = n-1$.  For a $j \in \{ 3, \ldots , k \}$, 
the verification of (\ref{eqn:35c}) goes  as follows: 
\begin{align*}
\sum_{ \substack{ ( \rho , v ) \in \NCord (n-1) \\
                      \mathrm{with} \ |J( \rho, v )| = j-1} } 
       t^{ Z_{n-1} ( \sigma , w ) } 
& = \sum_{ \substack{ ( \sigma , w ) \in \NCord (n-j+1) \\
                                \mathrm{with} \ |J( \sigma, w )| = 1} } 
t^{ Z_{n-j+1} ( \sigma , w ) + ( \pp_2 + \cdots + \pp_{j-1} ) }               \\
& \mbox{ $\ $ $\ $ (by Lemma \ref{lemma:33}, used with 
         $m = n-1$ and $\ell = j-1$) }                                     \\
& = t^{ \pp_2 + \cdots + \pp_{j-1} } \cdot 
             \sum_{ \substack{ ( \sigma , w ) \in \NCord (n-j+1) \\
                                \mathrm{with} \ |J( \sigma, w )| = 1} } 
        t^{ Z_{n-j+1} ( \sigma , w ) }                                     \\  
& = t^{\pp_2 + \cdots + \pp_{j-1}} \cdot (n-j+1) \, t^{\pp_1} \, L_{n-j} (t) 
    \ \mbox{ (by Lemma \ref{lemma:34}) }                                  \\
& = (n-j+1) \, t^{\pp_1 + \cdots + \pp_{j-1}} \, \, L_{n-j} (t), \ \mbox{ as required.}
\end{align*}

\vspace{6pt}

Finally, substituting the expressions from Claim 1 and Claim 2  
on the right-hand side of the Equation (\ref{eqn:35b}) yields the 
required formula (\ref{eqn:32a}) from the statement of the theorem.
\hfill $\square$
\end{ad-hoc-item}

\vspace{10pt}

\section{Statistics satisfying an
\texorpdfstring{$\NCord$}{NCord}-recursion of the first kind}
\label{section:4}

\setcounter{equation}{0}

\noindent
In this section we show how the recursion for combinatorial Laplace transforms
found in Theorem \ref{thm:32} applies to the statistics considered in Theorems \ref{thm:16}
and \ref{thm:17} of the Introduction, and we provide the proofs of these theorems.

\vspace{10pt}

\subsection{Block-counting random variables -- proof of Theorem \ref{thm:16}.} 
\label{subsection:4-1}

$\ $

\noindent
In this subsection we consider the random variables
$( Y_n : \NCord (n) \to \bR )_{n=1}^{\infty}$ introduced in Notation \ref{rem:15},
with $Y_n ( \pi, u ) := | \pi |$ for $( \pi , u ) \in \NCord (n)$, and we give
the proof of Theorem \ref{thm:16}. 

\vspace{6pt}

\begin{notation-and-remark}   \label{def:41}
For every $n \in \bN$, the combinatorial Laplace transform of $Y_n$ (defined as 
in Equation (\ref{eqn:31a})) will be denoted as $B_n$.  That is, we put
\begin{equation}  \label{eqn:41a}
B_n (t) := \sum_{( \pi , u) \in \NCord (n)} t^{| \pi |}, \ \ t \in (0, \infty ).
\end{equation}
Clearly, every $B_n$ is a polynomial function 
without constant term.  For example, direct inspection shows that
$B_1 (t) = t$, $B_2 (t) = 2t^2 + t$, $B_3 (t) = 6t^3 + 5t^2 + t$. 
\end{notation-and-remark} 

\vspace{6pt}

\begin{lemma}   \label{lemma:42}
The polynomials $B_n$ introduced above satisfy the recursion
\begin{equation}   \label{eqn:42a}
B_n (t) = (1+nt) \, B_{n-1} (t), \ \ n \geq 2,
\end{equation}
and are consequently given by the formula
\begin{equation}   \label{eqn:42b}
B_n (t) = t (1 +2t) \cdots (1 + nt), 
\ \mbox{ for $n \geq 2$ in $\bN$ and $t \in (0, \infty )$.}
\end{equation}
\end{lemma}

\begin{proof}  We saw in Example \ref{example:28} that the $Y_n$ are recursive of 
the first kind with input vector $(1) \in \bR^1$.  We use here that, equivalently, 
they are recursive of the first kind with input vector 
$\underline{r} = (r_1, r_2) = (1,0) \in \bR^2$, and we invoke
Equation (\ref{eqn:32a}) of Theorem \ref{thm:32} in connection to the latter input 
vector, to find that for every $n \geq 3$ we have:
\[
B_n (t) = (1 + n t^{r_1}) \, B_{n-1} (t) + (n-2+1) 
\, (t^{r_2} - 1) \, t^{r_1} \, B_{n-2} (t) = (1+ nt) B_{n-1} (t).
\]
This gives the formula (\ref{eqn:42a}), where for $n=2$ we check by direct inspection 
that $B_2 (t) = (1+2t) B_1 (t)$.  Equation (\ref{eqn:42b}) follows from (\ref{eqn:42a}) 
via an immediate induction argument.
\end{proof}

\vspace{6pt}

By using Lemma \ref{lemma:42}, it is easy to find the explicit 
expectation and variance formulas for the $Y_n$'s which 
were announced in Theorem \ref{thm:16} of the Introduction.  

\vspace{10pt}

\begin{ad-hoc-item}  \label{proof:43}
{\bf Proof of Theorem \ref{thm:16}.}
$1^o$ Differentiation in the product formula (\ref{eqn:42b}) leads to
\begin{equation}  \label{eqn:43x}
\frac{B_n ' (t)}{B_n (t)} = \frac{1}{t} + \frac{2}{1+2t} 
+ \frac{3}{1 + 3t}  + \cdots + \frac{n}{1 + nt }, \ \ t \in ( 0, \infty ),
\end{equation}
so upon invoking Equation (\ref{eqn:31b}) of Remark \ref{def:31} we get
\begin{equation}   \label{eqn:43y}
E[ \, Y_n \, ] =  \frac{B_n ' (1)}{B_n (1)} = 1 + \frac{2}{3} 
+ \frac{3}{4}  + \cdots + \frac{n}{n+1}
= n - \sum_{k=3}^{n+1} \frac{1}{k},
\end{equation}
leading to the formula for $E[Y_n]$ stated in Theorem \ref{thm:16}.1. 
The asymptotics stated in Equation (\ref{eqn:16b}) there follows
when we take into account that $H_n \approx \ln n + \gamma$. 

\vspace{6pt}

\noindent
$2^o$ Let us denote $R_n := B_n ' / B_n$.  Then direct calculation gives
\[
R_n + R_n' = \frac{B_n'}{B_n}
+ \frac{B_n'' B_n - ( B_n ' )^2}{B_n^2}
= \frac{B_n '' + B_n'}{B_n} - \Bigl[ \, \frac{B_n '}{B_n} \, \Bigr]^2,
\]
and comparing this against the second Equation (\ref{eqn:31b}) in
Remark \ref{def:31} we find that
\begin{equation}   \label{eqn:43z}
\mathrm{Var} (Y_n) = R_n (1) + R_n' (1).
\end{equation}

Now, Equation (\ref{eqn:43x}) gives an explicit formula 
for $R_n (t)$, and upon differentiating both sides of that equation we also 
find the explicit formula for $R_n ' (t)$: 
\[
R_n' (t) = - \frac{1}{t^2} - \frac{4}{(1+2t)^2} - \cdots - \frac{n^2}{(1+nt)^2},
\ \ t > 0.
\]
Substituting these two explicit formulas into (\ref{eqn:43z}) gives
\begin{align*}
\mathrm{Var} (Y_n) 
& = \Bigl( 1 + \frac{2}{3} + \cdots + \frac{n}{n+1} \Bigr)
- \Bigl( 1 + \frac{2^2}{3^2} + \cdots + \frac{n^2}{(n+1)^2} \Bigr) \\
& = \sum_{k=2}^n \Bigl( \, \frac{k}{k+1} - \frac{k^2}{(k+1)^2} \, \Bigr) 
  = \sum_{k=2}^n \frac{k}{(k+1)^2},
\end{align*}
which verifies the first equality stated in Theorem \ref{thm:16}.2.  For the second 
equality stated there, we write every $k/(k+1)^2$ as the difference between 
$1/(k+1)$ and $1/(k+1)^2$, to get
\begin{align*}
\mathrm{Var} (Y_n)
& = \Bigl( \frac{1}{3} + \cdots + \frac{1}{n+1} \Bigr)
- \Bigl( \frac{1}{3^2} + \cdots + \frac{1}{(n+1)^2} \Bigr) \\
& = \Bigl( H_{n+1} - 1 - \frac{1}{2} \Bigr) 
    - \Bigl( H_{n+1}^{(2)} - 1 - \frac{1}{4} \Bigr)  
  = H_{n+1} - H_{n+1}^{(2)} - \frac{1}{4},
\end{align*}
as required.
Finally, in order to obtain the asymptotics stated in (\ref{eqn:16d}),
we write 
\begin{align*}
\mathrm{Var} (Y_n) - \ln n
& = \bigl( H_n + \frac{1}{n+1} \bigr) - H_{n+1}^{(2)} - \frac{1}{4} - \ln n \\
& = ( H_n - \ln n ) - H_{n+1}^{(2)} - \frac{1}{4} + \frac{1}{n+1}.
\end{align*}
The latter quantity converges to $\gamma - \frac{\pi^2}{6} - \frac{1}{4} + 0$,
which gives the required $\approx$ formula.
\hfill $\square$
\end{ad-hoc-item}

\vspace{10pt}

\subsection{Block-counting tallied by size -- proof of Theorem \ref{thm:17}.}
\label{subsection:4-2}

$\ $

\noindent
In this subsection we consider the random variables $Y_n^{ ( \ell ) }$ defined in 
Equation (\ref{eqn:17a}) of Theorem \ref{thm:17} -- that is, $Y_n^{( \ell )} ( \pi, u )$ 
counts the blocks of cardinality $\ell$ of $\pi$, for $( \pi , u ) \in \NCord (n)$.
The main point of the subsection is to prove Proposition \ref{prop:46}, which repeats, 
with some added details, the statement of Theorem \ref{thm:17}.

\vspace{6pt}

\begin{notation-and-remark}   \label{def:44}
For every $\ell, n \in \bN$, the combinatorial Laplace transform of $Y_n^{( \ell )}$ 
(defined as in Equation (\ref{eqn:31a})) will be denoted as $B_n^{( \ell )}$. 
Since $Y_n^{ ( \ell ) }$ takes values in $\bN \cup \{ 0 \}$, every $B_n^{( \ell )}$ 
is a polynomial function.  For example, direct inspection shows that for $\ell = 1$ 
we have $B_1^{ (1) } (t) = t$, $B_2^{ (1) } (t) = 2t^2 + 1$, 
$B_3^{ (1) } (t) = 6t^3 + 5t^2 + 1$, while for $\ell = 2$ we have
$B_1^{ (2) } (t) = 1$, $B_2^{ (2) } (t) = t + 2$, $B_3^{ (2) } (t) = 5t + 7$.
\end{notation-and-remark} 

\vspace{6pt}

\begin{lemma}   \label{lemma:45}
Notation as above.

\vspace{6pt}

\noindent
$1^o$ The sequence of polynomials $\bigl( \, B_n^{(1)} \, \bigr)_{n=1}^{\infty}$ 
satisfies the recursion 
\begin{equation}   \label{eqn:45a}
B_n^{(1)} (t) = (nt+1) \, B_{n-1}^{(1)} (t) 
  + (n-1) \, (1-t) \, B_{n-2}^{(1)} (t), \ \ n \geq 3.
\end{equation}

\vspace{6pt}

\noindent
$2^o$ For every $\ell \geq 2$, the sequence of polynomials 
$\bigl( \, B_n^{( \ell )} \, \bigr)_{n = 1}^{\infty}$ satisfies the recursion 
\begin{equation}   \label{eqn:45b}
B_n^{( \ell )} (t) = (n+1) \, B_{n-1}^{( \ell )} (t)
+ (t-1) \, \Bigl( \, (n- \ell +1 ) \, B_{n - \ell}^{( \ell )} (t)
- (n- \ell) \, B_{n- \ell - 1}^{( \ell )} (t) \, \Bigr),
\end{equation}
holding for $n \geq \ell + 2$.  
\end{lemma}

\begin{proof}
Both Equations (\ref{eqn:45a}) and (\ref{eqn:45b}) are special cases of the general 
recursion formula (\ref{eqn:32a}) found in Theorem \ref{thm:32}, where we take into account
that (as seen in Example \ref{example:29}) the random variables 
$\bigl(Y_n^{(1)} \bigr)_{n=1}^{\infty}$ are recursive of the first kind with input vector 
$(1,-1)$, while for every $\ell \geq 2$ the random variables 
$\bigl( Y_n^{ ( \ell ) } \bigr)_{n=1}^{\infty}$ are recursive of the first kind with input 
vector $(0, \ldots, 0, 1, -1) \in \bR^{\ell + 1}$.
\end{proof}

\vspace{6pt}

\begin{proposition}   \label{prop:46}
Notation as above.

\vspace{6pt}

\noindent
$1^o$ One has that
$E[ \, Y_n^{(1)} \, ] = n - 2 H_n + \frac{10}{3} - \frac{3}{n+1},
\ \ n \geq 3,$
with the consequence that
\begin{equation}  \label{eqn:46a}
E[ \, Y_n^{(1)} \, ] \approx n - 2 \, \ln n + \bigl( \frac{10}{3} - 2 \gamma \bigr)
\ \ \mbox{ for $n \to \infty$.}
\end{equation}

\vspace{6pt}

\noindent
$2^o$ One has that
$E[ \, Y_n^{(2)} \, ] 
= H_n -\frac{51}{24} + \frac{6n-1}{2n(n+1)}, \ \ n \geq 4,$
with the consequence that
\begin{equation}  \label{eqn:46b}
E[ \, Y_n^{(2)} \, ] \approx \ln n - \bigl( \frac{51}{24} - \gamma \bigr)
\ \ \mbox{ for $n \to \infty$.}
\end{equation}

\vspace{6pt}

\noindent
$3^o$ For fixed $\ell \geq 3$ and $n \to \infty$, the sequence
$\bigl( \, E[ Y_n^{ ( \ell ) } ] \, \bigr)_{n=1}^{\infty}$
has a limit $c_{\ell} \in [0, \infty )$.
\vspace{6pt}

\noindent
$4^o$ Same as in Remark \ref{rem:18} of the Introduction, let
$Y_n^{ ( \geq 3 ) } : \NCord (n) \to \bR$ be the random variable which counts
all blocks of cardinality $\geq 3$,
\[
Y_n^{ ( \geq 3 ) } ( \pi, u ) 
:= \ \vline \, \{ V \in \pi \mid \, |V| \geq 3 \} \ \vline \ ,
\ \ \forall \, ( \pi, u ) \in \NCord (n).
\]
Then $E[ Y_n^{(\geq 3)} ] = \frac{7}{24} - \frac{2n-1}{2n (n+1)}$,
$n \geq 4$, with the consequence that
$\lim_{n \to \infty} E[ \, Y_n^{(\geq 3)} \, ] = \frac{7}{24}.$
\end{proposition}

\begin{proof} $1^o$ When we differentiate both sides of Equation (\ref{eqn:45a}) and 
evaluate the result at $t=1$, we find that
\[  
[ B_n^{(1)} ]' (1) 
= n B_{n-1} (1) + (n+1) [ B_{n-1}^{(1)} ]' (1)
- (n-1) \, B_{n-2}^{(1)} (1), \\ 
\]
\begin{equation}  \label{eqn:46c}
= n \cdot \frac{n!}{2} + (n+1) [ B_{n-1}^{(1)} ]' (1)
- (n-1) \, \frac{(n-1)!}{2}, \ \ n \geq 3,
\end{equation}
where the second equality takes into account that
$B_{n-1} (1) = | \NCord (n-1) | = n!/2$, and likewise for 
$B_{n-2} (1) = (n-1)!/2$.  But, as noted in Equation (\ref{eqn:31b}) 
from Remark \ref{def:31}, the derivative at $1$ of a combinatorial Laplace transform is 
related to the expectation of the corresponding random variable, in a way which allows
us to write
\begin{equation}  \label{eqn:46d}
[ B_n^{(1)} ]' (1) = \frac{(n+1)!}{2} \, E[ \, Y_n^{(1)} \, ],
\ \ [ B_{n-1}^{(1)} ]' (1) = \frac{n!}{2} \, E[ \, Y_{n-1}^{(1)} \, ].
\end{equation}
Substituting (\ref{eqn:46d}) into (\ref{eqn:46c}) leads to the recursion 
\begin{equation}    \label{eqn:46e}
E[ \, Y_n^{(1)} \, ] = E[ \, Y_{n-1}^{(1)} \, ] 
+ 1 + \frac{1}{n} - \frac{3}{n+1}, \ \ n \geq 3.
\end{equation}
From (\ref{eqn:46e}), an easy induction on $n$ leads to the explicit
formula for $E[ \, Y_n^{(1)} \, ]$ stated in the proposition.  This, in turn,
immediately entails the asymptotics (\ref{eqn:46a}).

\vspace{6pt}

\noindent
$2^o \& 3^o$  For every $\ell \geq 2$, one can perform the same steps as indicated
in the proof of $1^o$ above, but starting now from Equation (\ref{eqn:45b}) of
Lemma \ref{lemma:45}.2.  Upon performing the necessary differentiation, followed 
by setting $t=1$ and by replacing the derivatives $[ B_n^{( \ell )} ]' (1)$, 
$[ B_{n-1}^{( \ell )} ]' (1)$ in terms of the expectations
$E[ Y_n^{( \ell )} ], \, E[ Y_{n-1}^{ (\ell) } ]$, we find that
\[
[ B_n^{( \ell )} ]' (1) = (n+1) \, [ B_{n-1}^{( \ell )} ]' (1)
+ (n- \ell + 1) \cdot \frac{ (n - \ell + 1)!}{2} 
- (n- \ell) \cdot \frac{ (n - \ell)!}{2},
\]
and hence that
\begin{equation}  \label{eqn:46f}
E[Y_n^{(\ell)}] = E[Y_{n-1}^{(\ell)}] + \frac{(n - \ell )!}{(n+1)!} 
\bigl[ (n - \ell + 1)^2 - ( n- \ell) \bigr], \ \ n \geq \ell + 2.
\end{equation}

\vspace{6pt}

For the rest of the discussion, the case $\ell = 2$ goes differently from 
$\ell \geq 3$.

\vspace{6pt}

For $\ell \geq 3$, one only needs to note that (\ref{eqn:46f}) implies the inequality
\[
\vline \ E[Y_n^{(\ell)}] - E[Y_{n-1}^{(\ell)}] \ \vline
\leq \frac{(n - \ell )!}{(n+1)!} 
\bigl[ (n - \ell + 1) (n - \ell + 2) \bigr] \leq \frac{1}{n(n+1)}.
\]
Since $\sum_{n=1}^{\infty} \frac{1}{n(n+1)} < \infty$, this implies that the sequence of 
$E[Y_n^{(\ell)}]$'s converges for $n \to \infty$.

\vspace{6pt}

For $\ell = 2$, the recursion (\ref{eqn:46f}) reduces to
\begin{align*}
E[ \, Y_n^{(2)} \, ] 
& = E[ \, Y_{n-1}^{(2)} \, ] + \frac{n^2 - 3n + 3}{(n-1)n(n+1)}
& = E[ \, Y_{n-1}^{(2)} \, ] +  
 \Bigl[ \, \frac{1}{2} \cdot \frac{1}{n-1} \cdot - 3 \cdot \frac{1}{n} 
         + \frac{7}{2} \cdot \frac{1}{n+1} \, \Bigr],
\end{align*}
holding for $n \geq 4$.  Telescoping the latter formula leads to
\begin{equation}   \label{eqn:46g}
E[ \, Y_n^{(2)} \, ] =  E[ \, Y_3^{(2)} \, ] - \frac{11}{24} 
+ \Bigl[ \, \frac{1}{5} + \cdots + \frac{1}{n-1} \, \Bigr]
+ \Bigl[ \, \frac{1}{2} \cdot \frac{1}{n} + \frac{7}{2} \cdot \frac{1}{n+1} \, \Bigr].
\end{equation}
Finally, upon completing the harmonic number $H_n$ in (\ref{eqn:46g}) 
one finds the concrete expression for $E[Y_n^{(2)}]$ stated in $2^o$, 
which has the asymptotics (\ref{eqn:46b}) as an immediate consequence.

\vspace{6pt}

\noindent
$4^o$ This follows by taking expectations in
the relation $Y_n^{ ( \geq 3 ) } = Y_n - Y_n^{(1)} - Y_n^{(2)}$
(with $Y_n$ as in Theorem \ref{thm:16}),
and then by invoking the formulas for expectations obtained in Theorem \ref{thm:16}.1 
and in $1^o, 2^o$ of the present proposition.
\end{proof}

\vspace{10pt}

\begin{remark}  \label{rem:47}
$1^o$ The limits $c_{\ell}$ mentioned in Proposition \ref{prop:46}.3 can be 
explicitly computed by doing partial fraction decomposition in the formula
for $E[ Y_n^{( \ell )} ] - E[ Y_{n-1}^{( \ell )} ]$ which comes out 
of Equation (\ref{eqn:46f}).  For instance: for $\ell = 3$, this partial 
fraction decomposition is
\[
E[Y_n^{(3)}] -  E[Y_{n-1}^{(3)}] 
= \frac{1}{6} \cdot \frac{1}{n-2} - \frac{3}{2} \cdot \frac{1}{n-1} 
  + \frac{7}{2} \cdot \frac{1}{n} - \frac{13}{6} \cdot \frac{1}{n+1},
\ \ n \geq 5.
\]
We leave it as an exercise to the patient reader to check that,
upon telescoping the latter expression (in order to reduce $E[Y_n^{(3)}]$ 
to $E[Y_{n-1}^{(3)}]$, then to $E[Y_{n-2}^{(3)}] , \ldots ,$ all the way 
to $E[Y_{4}^{(3)}] = 1/10$), one finds that 
$\lim_{n \to \infty} E[Y_n^{(3)}]= 23/90$.

\vspace{6pt}

\noindent
$2^o$ The formula found in part $4^o$ of the preceding proposition can be obtained
by the same method as shown in the proof of parts $1^o - 3^o$.  Indeed, since the 
sequence of random variables $\bigl( \, Y_n^{( \geq 3)} \, \bigr)_{n=1}^{\infty}$ is 
recursive of the first kind with input vector $\ur = (0,0,1) \in \bR^3$, 
Theorem \ref{thm:32} gives us that the corresponding combinatorial Laplace transforms
$B_n^{ ( \geq 3 ) }$ satisfy the recursion
\[
B_n^{ ( \geq 3 ) } (t) 
= (n+1) \, B_{n-1}^{ ( \geq 3 ) } (t) 
+ (t-1) \, (n-2) \, B_{n-3}^{ ( \geq 3 ) } (t), \ \ n \geq 4.
\]
This further implies a recursion for the expectations $E[ Y_n^{( \geq 3 )} ]$, and 
entails the formulas stated in Proposition \ref{prop:46}.4.
\end{remark}

\vspace{10pt}

\section{An aside to Section \ref{subsection:4-1}: 
moments of the monotonic Poisson process}
\label{section:5}

\setcounter{equation}{0}

\noindent
In this section we do a brief detour from the main line of the paper, to 
point out that the considerations on the block-counting statistic from 
Section \ref{subsection:4-1} can be cast in a non-commutative probability 
setting, where they become considerations on {\em monotonic cumulants}. 
In order to review the latter notion, 
we will use (as commonly done by people working in this area) a simplified 
algebraic version for what is a probability distribution on the real line. 
We want to look at such distributions where every polynomial is integrable,
which makes it convenient to go with the following definition.

\vspace{6pt}

\begin{definition}  \label{def:51}
By {\em algebraic distribution} we will mean a linear functional
$\mu : \bC [X] \to \bC$ with the property that $\mu (1) = 1$. For such a
functional, it is customary that the number $\mu (X)$ is called 
{\em expectation} of $\mu$; and in general, for every $n \in \bN$, the number
$\mu (X^n) \in \bC$ will be referred to as the {\em moment of order $n$} of $\mu$.
\end{definition}

\vspace{6pt}

\begin{remark-and-definition}  \label{rem:52}
Let $\mu : \bC[X] \to \bC$ be an algebraic distribution. There exists a  
sequence of complex numbers $(\, c_n ( \mu ) \, )_{n=1}^{\infty}$, uniquely determined, 
such that
\begin{equation}   \label{eqn:52a}
\mu (X^n) = \sum_{ (\pi, u) \in \NCord (n) } \, \Bigl[ \, \frac{1}{ | \pi | ! } \cdot 
\, \prod_{V \in \pi} c_{|V|} ( \mu ) \, \Bigr], \ \forall \, n \in \bN.
\end{equation}
The numbers $c_n ( \mu )$ are called {\em monotonic cumulants} of $\mu$, and 
Equation (\ref{eqn:52a}) goes under the name of 
``moment-cumulant formula'' (for monotonic cumulants).
Upon grouping terms on the right-hand side, we can also write it in the form:
\begin{equation}   \label{eqn:52b}
\mu (X^n) = \sum_{\pi \in NC(n)} \, \Bigl[ \, \frac{\mathrm{monord} ( \pi )}{ | \pi | ! } \cdot 
\, \prod_{V \in \pi} c_{|V|} ( \mu ) \, \Bigr], \ \, n \in \bN,
\end{equation}
where $\mathrm{monord} ( \pi )$ stands for the number of monotonic orderings of
$\pi \in NC(n)$.  The ratio $\mathrm{monord} ( \pi ) /  | \pi | !$ can be thought of as 
the probability that a randomly picked ordering of the blocks of $\pi$ is monotonic. This 
ratio is equal to $1$ for an interval partition (when every $V \in \pi$ is a 
sub-interval of $\{ 1, \ldots , n \}$), but is $<1$ when nestings 
among the blocks of $\pi$ impose restrictions on what orderings of $\pi$ are monotonic.

The existence of a sequence $( \, c_n ( \mu ) \, )_{n=1}^{\infty}$,
uniquely determined, such that Equation (\ref{eqn:52a}) holds, is easily seen by isolating 
in the sum on the right-hand side the unique term which has 
$\pi = \bigl\{ \, \{1, \ldots , n\} \, \bigr\} =:1_n$ (the partition with only one block).
That special term is equal to $c_n ( \mu )$, which allows us to re-write (\ref{eqn:52b}) 
in the form 
\begin{equation}   \label{eqn:52c}
c_n ( \mu ) = \mu (X^n) - \sum_{ \substack{ \pi \in NC(n), \\ \pi \neq 1_n} }
 \Bigl[ \, \frac{\mathrm{monord} ( \pi )}{ | \pi | ! } \cdot 
\, \prod_{V \in \pi} c_{|V|} ( \mu ) \, \Bigr], 
\end{equation}
and then use induction in order to write an explicit formula for $c_n ( \mu )$ in terms 
of the moments $\mu (X^k), \ 1 \leq k \leq n$.
For instance, for the first few values of $n$ we get:
\[
\mu (X) = c_1 ( \mu ), \ \mu (X^2) = c_2 ( \mu ) + \bigl( \, c_1 ( \mu ) \, \bigr)^2,
\ \mu (X^3) = c_3 (\mu) + \frac{5}{2} c_1 ( \mu) \, c_2 ( \mu ) + \bigl( \, c_1 ( \mu ) \, \bigr)^3,
\]
which gives 
\[
c_1 ( \mu ) = \mu (X), \ c_2 ( \mu ) = \mu (X^2) - \bigl( \, \mu (X) \, \bigr)^2,
\ c_3 ( \mu ) = \mu (X^3) - \frac{5}{2} \mu (X) \, \mu (X^2) + \frac{3}{2} \bigl( \, \mu (X) \, \bigr)^3.
\]

Another easily derived consequence of (\ref{eqn:52b}) and (\ref{eqn:52c}) is this:
for every sequence of complex numbers $( \alpha_n )_{n=1}^{\infty}$ there 
exists an algebraic distribution $\mu$, uniquely determined, such that
$c_n ( \mu ) = \alpha_n$ for all $n \in \bN$.  In other words, it is possible to 
define an algebraic distribution $\mu$ via a prescription of what are the monotonic 
cumulants $c_n ( \mu )$.  This can in particular be used (following a well-established 
pattern from other brands of non-commutative probability) in order to pinpoint the 
monotonic version of Poisson distribution, as follows.
\end{remark-and-definition}

\vspace{6pt}

\begin{notation}   \label{def:53}
For every $\alpha \in [ 0, \infty )$ we will let $\nu_{\alpha} : \bC [X] \to \bC$ be the 
algebraic distribution defined via the requirement that it has monotonic cumulants
\begin{equation}   \label{eqn:53a}
c_n ( \nu_{\alpha} ) = \alpha, \ \ \forall \, n \in \bN.
\end{equation}
The family of distributions $\bigl( \, \nu_{\alpha} \, \bigr)_{\alpha \in [0, \infty )}$
goes under the name of {\em monotonic Poisson process}.
\end{notation}

\vspace{6pt}

In \cite{Be2006}, Belton gave a nice description of the moments of the distributions 
$\nu_{\alpha}$, in terms of a double-array of integers $J_k^{(n)}$ (not depending on
$\alpha$) which satisfy a simple recursive relation.  We repeat this here, and observe 
how the said recursive relation can be read by looking at the $\NCord$-tree.

\vspace{6pt}

\begin{proposition}  \label{prop:54}
{\em (\cite[Section 2]{Be2006}.)}
The moments of the monotone Poisson process are  
\begin{equation}   \label{eqn:54a}
\nu_{\alpha} (X^n) = \sum_{k=1}^n J_k^{(n)} \cdot \frac{\alpha^k}{k!},
\ \ \forall \, \alpha \in [0, \infty ), \,  n \in \bN, 
\end{equation}
with the double array $\bigl\{ J_k^{(n)} \mid n,k \in \bN, \, n \geq k \bigr\}$ 
defined by initializing
\begin{equation}   \label{eqn:54b}
J_1^{(n)} = 1 \mbox{ and } J_n^{(n)} = n! \, , \ \ \forall \, n \in \bN
\end{equation}
(in particular $J_1^{(1)} = 1$, $J_1^{(2)} = 1$, $J_2^{(2)} = 2$), 
and then by requiring that
\begin{equation}   \label{eqn:54c}
J_k^{(n)} = J_k^{(n-1)} + n J_{k-1}^{(n-1)}, 
\ \mbox{ for all $n \geq 3$ and $2 \leq k \leq n-1$.}
\end{equation}
\end{proposition}

\begin{proof} The uniqueness of the numbers $J_k^{(n)}$ is clearly implied 
by the requirements made in (\ref{eqn:54b})+ (\ref{eqn:54c}).  For the
existence part of the proposition, it is convenient to go by putting
\begin{equation}   \label{eqn:54d}
J_k^{(n)} := \ \vline \bigl\{ \, ( \pi, u ) \in \NCord (n) \mid
| \pi | = k \, \bigr\} \ \vline \ , \mbox{ for $1 \leq k \leq n$ in $\bN$.}
\end{equation}
Then for every $\alpha \in [0, \infty )$, the moment-cumulant formula (\ref{eqn:52a})
used for $\nu_{\alpha}$ gives:
\begin{align*}
\nu_{\alpha} (X^n) 
& = \sum_{ ( \pi, u ) \in \NCord (n)}  
\frac{1}{ | \pi | !} \cdot \alpha^{ | \pi | }   \\
& = \sum_{k=1}^n \Bigl[ 
\, \sum_{ \substack{ (\pi, u) \in \NCord (n) \\ \mathrm{with} \ | \pi | = k } }
\frac{ \alpha^k }{k!} \, \Bigr] \, = 
\, \sum_{k=1}^n J_k^{(n)} \cdot \frac{ \alpha^k }{k!}, \ n \in \bN,
\end{align*}
showing that (\ref{eqn:54a}) holds.  It is also immediate that the numbers defined
by (\ref{eqn:54d}) satisfy the initial conditions (\ref{eqn:54b}), since 
$J_1^{(n)}$ counts the unique ordering of the partition $1_n$, while 
$J_n^{(n)}$ counts the $n!$ orderings of the partition 
$0_n := \bigl\{ \, \{1\}, \, \{2 \}, \ldots , \{n\} \, \bigr\} \in NC(n)$.
We are thus left to verify that the $J_k^{(n)}$'s from (\ref{eqn:54d}) satisfy 
the recursion relation (\ref{eqn:54c}).  

For the rest of the proof we fix $n \geq 3$ and $2 \leq k \leq n-1$, for which we 
will verify that (\ref{eqn:54c}) holds. Towards that end we decompose the set 
$\bigl\{ ( \pi, u ) \in \NCord (n) \mid  \, | \pi | = k \bigr\}$ according to whether
the block $J( \pi, u )$ of $\pi$ is a singleton-block or not. We get
\[
\bigl\{ \, ( \pi, u ) \in \NCord (n) \mid
| \pi | = k \, \bigr\} = S_0 \cup S_1 \cup \cdots \cup S_n, 
\mbox{ disjoint union,}
\]
where 
$S_0 := \bigl\{ \, ( \pi, u ) \in \NCord (n) \mid
| \pi | = k \mbox{ and } | \, J( \pi, u ) \, | \geq 2 \bigr\},$
and where for every $1 \leq m \leq n$ we put 
$S_m := \bigl\{ \, ( \pi, u ) \in \NCord (n) \mid
| \pi | = k \mbox{ and } J( \pi, u ) = \{ m \} \bigr\}.$
We leave it as a straightforward exercise to the reader to use the description 
of the parent map $\parent$ of the $\NCord$-tree (cf.~Definition \ref{def:22}) in order to 
check that, for every $1 \leq m \leq n$, the map $\parent$ sends $S_m$ bijectively onto 
$\bigl\{ \, ( \rho, v ) \in \NCord (n-1) \mid | \rho | = k-1 \, \bigr\}$, with the 
consequence that 
$| S_m | = \ \vline \
\bigl\{ \, ( \rho, v ) \in \NCord (n-1) \mid | \rho | = k-1 \, \bigr\} \ \vline
\ = J_{k-1}^{(n-1)}.$
A similar exercise   shows that $\parent$ sends $S_0$ bijectively onto 
$\bigl\{ \, ( \rho, v ) \in \NCord (n-1) \mid | \rho | = k \, \bigr\}$, 
with the consequence that $| S_0 | = J_k^{(n-1)}$.
By putting these things together we find that 
\[
J_k^{(n)} 
= \ \vline \ \bigl\{ \, ( \pi, u ) \in \NCord (n) \mid | \pi | = k \bigr\} 
\ \vline \ = |S_0| + |S_1| + \cdots + |S_n| 
= J_k^{(n-1)} + n J_{k-1}^{(n-1)},
\]
as required.
\end{proof}

\vspace{6pt}

In \cite[Section 2]{Be2006} it is explained that the numbers $J_k^{(n)}$
defined by the recursion (\ref{eqn:54c}) are part of 
a larger family of ``generalized Stirling numbers of first kind''. 
The first few lines of the double-array $[ J_k^{(n)} ]_{n,k}$ are shown in 
Figure 4 below.  It is actually possible to give an explicit formula for an 
individual $J_k^{(n)}$, as shown in Equation (\ref{eqn:55b}) of the next 
proposition.  We note that in the analytic approach of Belton, the statements 
of Propositions \ref{prop:54} and \ref{prop:55} came in reverse order, with 
the formula (\ref{eqn:55b}) derived first, in the earlier paper \cite{Be2005}.

\vspace{10pt}

\begin{center}

$\begin{array}{cccccccc}
      &  \vline  &(k=1) &(k=2) & (k=3) & (k=4) & (k=5) & (k=6)     \\ \hline
      &  \vline  &      &      &       &        &        &          \\
(n=1) &  \vline  &  1   &      &       &        &        &          \\
(n=2) &  \vline  &  1   &   2  &       &        &        &          \\
(n=3) &  \vline  &  1   &   5  &    6  &        &        &          \\
(n=4) &  \vline  &  1   &   9  &   26  &   24   &        &          \\
(n=5) &  \vline  &  1   &  14  &   71  &  154   &   120  &          \\
(n=6) &  \vline  &  1   &  20  &  155  &  580   &  1044  &  720
\end{array}$

\vspace{10pt}

Figure 4. 
{\em The numbers $J_k^{(n)}$ for $n \leq 6$ and $k \in \{ 1, \ldots , n \}$.  }
\end{center}

\vspace{6pt}

\begin{proposition}   \label{prop:55}
{\em (\cite[Section 4]{Be2005}.)}
Let $n \geq 2$ be in $\bN$, and consider the numbers $J_1^{(n)}$, 
$J_2^{(n)}, \ldots , J_n^{(n)}$ introduced in Proposition \ref{prop:54}.  
One has the polynomial factorization
\begin{equation}   \label{eqn:55a}
J_1^{(n)} t + J_2^{(n)} t^2 + \cdots + J_n^{(n)} t^n
= t( 1 + 2t ) \cdots ( 1 + nt ), 
\end{equation} 
with the consequence that one can explicitly write
\begin{equation}   \label{eqn:55b}
J_k^{(n)} = \sum_{2 \leq j_1 < \cdots < j_{k-1} \leq n}
j_1 \cdots j_{k-1}, \ 2 \leq k \leq n.
\end{equation} 
\end{proposition}

\begin{proof}  It is convenient to take advantage of the framework 
developed in Section \ref{subsection:4-1}, which we do by using 
the formula for $J_k^{(n)}$ provided by Equation (\ref{eqn:54d}) 
from the proof of Proposition \ref{prop:54}.  We get:
\[
\sum_{k=1}^n J_k^{(n)} t^k
= \sum_{k=1}^n \Bigl[ \, \sum_{ \substack{ (\pi, u) \in \NCord (n) \\ 
                                     \mathrm{with} \ | \pi | = k } }
t^{ | \pi | }  \, \Bigr] \, = \, \sum_{(\pi,u) \in \NCord (n)} t^{ | \pi | }
= B_n (t),
\]
where $B_n$ is the polynomial introduced in Notation \ref{def:41}. 
Lemma \ref{lemma:42} then gives us the factorization (\ref{eqn:55a}), 
and the formula (\ref{eqn:55b}) for $J_k^{(n)}$ follows when we extract 
the coefficient of $t^k$ on the two sides of (\ref{eqn:55a}).
\end{proof}

\vspace{10pt}

\section{Laplace transforms for the
\texorpdfstring{$\NCord$}{NCord}-recursion of the second kind}
\label{section:6}

\setcounter{equation}{0}

\noindent
We now take on a development analogous to the one from
Sections \ref{section:3} and \ref{section:4}, in connection to the 
$\NCord$-recursion property ``of the second kind'' introduced in 
Section \ref{subsection:2-3}.  More precisely: we first find, in 
Theorem \ref{thm:61}, how this $\NCord$-recursion property yields 
a recursion for the combinatorial Laplace transforms of the random variables 
in question.  Then we see how Theorem \ref{thm:61} can be applied in the setting 
of Example \ref{example:212}, in order to determine the expected number of outer 
blocks of a random $( \pi , u ) \in \NCord (n)$.

\vspace{6pt}

\begin{theorem}    \label{thm:61}
Let $\alpha, \beta, q \in \bZ$ and let $( Z_n : \NCord (n) \to \bR )_{n=1}^{\infty}$ be 
a sequence of random variables which is recursive of the second kind with input 
$( \alpha , \beta , q)$, in the sense of Definition \ref{def:211}.  For every $n \in \bN$, 
let $L_n : ( 0, \infty ) \to \bR$ be the combinatorial Laplace transform of $Z_n$, defined 
by Equation (\ref{eqn:31a}) of Definition \ref{def:31}.  Then one has the recursion:
\begin{equation}  \label{eqn:61a}
L_n (t) =  \Bigl( q \, t^{\alpha} + ((n+1)-q) \, t^{\beta} \Bigr) \cdot L_{n-1} (t)
         + \bigl( t^{\alpha + 1}-t^{\beta +1} \bigr) \cdot L_{n-1}' (t),
\end{equation}
holding for $n \geq 2$ and $t \in (0, \infty )$.
\end{theorem}

\begin{proof}
Consider an $n \geq 2$ in $\bN$ and a $t \in (0, \infty)$, for which we will verify 
(\ref{eqn:61a}).  Following the same idea as for the proof of Theorem \ref{thm:32}, 
we organize the sum defining $L_n (t)$ as
\begin{equation}  \label{eqn:61b}
L_n (t) = \sum_{( \rho , v) \in \NCord (n-1)} \Bigl[ 
\, \sum_{ ( \pi, u ) \in C( \rho , v) } t^{Z_n ( \pi , u )} \, \Bigr] ,
\end{equation}
where $C( \rho , v )$ denotes the set of children of $( \rho , v )$ in the $\NCord$-tree.

Let us fix for the moment a $( \rho , v ) \in \NCord (n-1)$, 
and examine the inner sum corresponding to $( \rho , v )$ in (\ref{eqn:61b}).
We break this sum into two by using the set $C_o ( \rho , v ) \subseteq C( \rho , v )$ 
provided by the definition of the recursion of the second kind (cf.~(\ref{eqn:211a}) there), 
and get: 
\begin{align*}
\sum_{ ( \pi, u ) \in C( \rho , v ) } t^{Z_n ( \pi,u )}
& = \sum_{ ( \pi, u ) \in C_o( \rho , v) } t^{Z_n ( \pi , u )} 
+ \sum_{ ( \pi, u ) \in C( \rho , v) \setminus C_o( \rho , v ) } t^{Z_n ( \pi , u )} \\
& = | C_o ( \rho , v ) | \cdot t^{Z_{n-1} ( \rho , v ) + \alpha }
+ | C( \rho, v ) \setminus C_o ( \rho , v ) | \cdot t^{Z_{n-1} ( \rho , v ) + \beta }  \\
& = ( Z_{n-1} ( \rho, v ) + q ) \cdot t^{Z_{n-1} ( \rho , v ) + \alpha }
+ ( (n+1) - Z_{n-1}( \rho , v ) -q ) \cdot t^{Z_{n-1} ( \rho , v ) + \beta }     \\
& = t^{Z_{n-1} ( \rho , v )} \, \Bigl(  \, (t^{\alpha} - t^{\beta}) \cdot Z_{n-1} ( \rho, v ) 
    + \bigl( \, q t^{\alpha } + ((n+1) - q) t^{\beta} \, \bigr) \,\Bigr).
\end{align*}

Returning now to Equation (\ref{eqn:61b}), we see that the calculation shown in the 
preceding paragraph decomposes $L_n (t) = \Sigma ' + \Sigma ''$, where
\begin{align*}
\Sigma ' 
& =  \sum_{(\rho,v) \in \NCord (n-1)} \, t^{Z_{n-1}(\rho,v)} \cdot
\bigl( t^{\alpha} - t^{\beta} \bigr) \, Z_{n-1} ( \rho,v )  \\
& = \bigl( t^{\alpha + 1} - t^{\beta + 1} \bigr) \cdot 
\sum_{(\rho,v) \in \NCord (n-1)} \, Z_{n-1} ( \rho, v ) \, t^{Z_{n-1}(\rho,v) - 1}  \\
& = \bigl( t^{\alpha + 1} - t^{\beta + 1} \bigr) \cdot L_{n-1}' (t),
\end{align*}
and
\begin{align*}
\Sigma '' 
& =  \sum_{(\rho,v) \in \NCord (n-1)} \, t^{Z_{n-1}(\rho,v)} \cdot 
\bigl( q \, t^{\alpha} + ((n+1) - q) \, t^{\beta} )  \\
& = \bigl( q \, t^{\alpha} + ((n+1) - q) \, t^{\beta} ) \cdot L_{n-1} (t).
\end{align*}
The expression $\Sigma ' + \Sigma ''$ gives the right-hand side of 
the required formula (\ref{eqn:61a}).
\end{proof}

\vspace{6pt}

\begin{corollary}    \label{cor:62}
In the framework and notation of Theorem \ref{thm:61}, the sequence of expectations
$\bigl( \, E[Z_n] \, \bigr)_{n=1}^{\infty}$ satisfies the recursion
\begin{equation}  \label{eqn:62a}
E[ Z_n ] = \frac{(n+1) + ( \alpha - \beta )}{n+1} \cdot E[ Z_{n-1} ]
+ \frac{\alpha q  + \beta ((n+1)-q)}{n+1}, \ \ n \geq 2.
\end{equation}
\end{corollary}

\begin{proof}
Differentiating both sides of Equation (\ref{eqn:61a}) and then evaluating them 
at $t=1$ gives
\begin{equation}  \label{eqn:62b}
L_n' (1) = \bigl( \, \alpha q + \beta ((n+1) - q) \, \bigr) \cdot \frac{n!}{2}
+ \bigl( \, (n+1) + ( \alpha - \beta ) \, \bigr) \cdot L_{n-1}' (1).
\end{equation}
We know (as noted in Remark \ref{def:31}) that in (\ref{eqn:62b}) we can substitute 
$L_n' (1) = E[ Z_n ] \cdot (n+1)!/2$ and $L_{n-1}' (1) = E[ Z_{n-1} ] \cdot n!/2$;
doing so leads to the formula (\ref{eqn:62a}) stated in the corollary.
\end{proof}

\vspace{6pt}

\begin{corollary}    \label{cor:63}
Let $\bigl( \, \OuT_n : \NCord (n) \to \bN \, \bigr)_{n=1}^{\infty}$ be the 
sequence of random variables considered in Example \ref{example:212}, where 
$\OuT_n ( \pi , u )$ counts the outer blocks of $\pi$, for $( \pi, u ) \in \NCord (n)$.
One has that
\begin{equation}  \label{eqn:63a}
E[ \OuT_n ] = \frac{2n+1}{3}, \ \ n \in \bN.
\end{equation}
\end{corollary}

\begin{proof}
We know from Example \ref{example:212} that the sequence of $\OuT_n$'s
is recursive of the second kind, with parameters $\alpha=1$, $\beta = 0$ and $q=1$.
We can thus use the formula (\ref{eqn:62a}) provided by the preceding corollary. 
This says, for the case at hand, that we have
\begin{equation}  \label{eqn:63b}
E[ \OuT_n ] = \bigl( 1 + \frac{1}{n+1} \bigr) \cdot E[ \OuT_{n-1} ]
+ \frac{1}{n+1}, \ \ n \geq 2.
\end{equation}
From (\ref{eqn:63b}), an easy induction (initialized with $E[ \OuT_1 ] = 1$) gives 
the formula (\ref{eqn:63a}).
\end{proof}

\vspace{1cm}

\section{The \texorpdfstring{$\NCord_2$}{NCord2}-tree, and
some combinatorial statistics related to it}
\label{section:7}

\setcounter{equation}{0}

\noindent
We now turn our attention to sets 
of monotonically ordered non-crossing {\em pair-partitions},
\[
\NCord_2 (2n) := \Bigl\{ ( \pi , u ) 
\begin{array}{lr}
\vline  & \pi \in NC_2 (2n)
\mbox{ and } u : \pi \to \{1, \ldots , n\},  \\
\vline  & \mbox{monotonic ordering} \end{array} \Bigr\}, \ \ n \in \bN.
\]
In this section we follow up on the discussion around these sets 
made in Subsection \ref{subsection:1-3} of the Introduction, and we start with
the fact that one has a natural tree structure on the union
$\NCord_2 := \sqcup_{n=1}^{\infty} \NCord_2 (2n)$.

\vspace{10pt}

\subsection{Description of the \texorpdfstring{\boldmath{$\NCord_2$}}{NCord2}-tree.}
\label{subsection:7-1}

\begin{remark}   \label{rem:71}
Since $\NCord_2 (2n) \subseteq \NCord (2n)$, we can and will continue to 
refer to Notation \ref{def:21}.1 and write ``$J( \pi, u )$'' 
for the block of a pair-partition $\pi \in NC_2 (2n)$ which has the maximal 
label in the monotone ordering $u$ of $\pi$.  The specifics of the $( \pi, u )$ 
considered here ensure that $| J( \pi, u ) | = 2$ and $u( \, J( \pi, u ) \, ) = n$.  
Since we additionally know (cf.~Remark \ref{rem:12}) that $J( \pi, u )$ is 
a sub-interval of $\{1, \ldots , 2n \}$, we can actually be certain to have here
$J( \pi, u ) = \{ m, m+1 \}$ for some $1 \leq m \leq 2n-1$. 

\vspace{6pt}

\noindent
$2^o$ We will also continue to use the ``restrict-and-relabel'' operation 
$\pi \mapsto \pi_W$ introduced in Notation \ref{def:21}.2, making sure 
that we only apply it to situations where both $\pi$ and $\pi_W$ belong to 
$\NCord_2$.   

\vspace{6pt}

\noindent
$3^o$ Similarly to the development shown for the $\NCord$-tree in Section \ref{subsection:2-1}, 
we will describe the edges of the $\NCord_2$-tree by indicating how one finds the 
$\NCord_2$-parent of a $( \pi, u ) \in \NCord_2 (2n)$, for $n \geq 2$.  In order to distinguish 
from the similar terminology introduced in Definition \ref{def:22}, we will refer to the 
$\NCord_2$-parent of such a $(\pi , u )$ by denoting it as $\pairparent ( \pi, u )$.  This 
notation fits with the fact that, if $(\pi,u)$ was to be viewed as  vertex of the $\NCord$-tree, 
then $\pairparent ( \pi, u ) \in \NCord_2 (2n-2)$ would come out as ``the grandparent'' (parent of
the parent) of $( \pi , u )$ in that context.
\end{remark}

\vspace{6pt}

\begin{definition}   \label{def:72}
Let $( \pi , u )$ be a monotonically ordered pair-partition in $\NCord_2 (2n)$, where 
$n \geq 2$.  Consider the pair $J( \pi, u ) = \{ m, m+1 \} \in \pi$, and consider  
the restrict-and-relabel partition
\[
\rho := \pi_{ \{1, \ldots , 2n \} \setminus \{ m, m+1 \} } \in NC_2 (2n-2). 
\]
It is easily checked that the monotonic order $u : \pi \to \{ 1, \ldots , n \}$ 
induces a monotonic order $v : \rho \to \{ 1, \ldots , n-1 \}$, defined by the formula:
\begin{equation}   \label{eqn:72a}
\left\{   \begin{array}{l} 
v \bigl( \{ p,q\} \bigr) 
     := u \Bigl( \, \bigl\{ \phi (p), \phi (q) \bigr\} \, \Bigr) 
     \mbox{ for $\{ p,q \} \in \rho$, where $\phi$ is the unique}      \\
\mbox{increasing bijection from $\{1 , \ldots , 2n-2 \}$ 
onto $\{1, \ldots , 2n \} \setminus \{ m, m+1 \}$. } 
\end{array} \right.
\end{equation}
The monotonically ordered pair-partition $( \rho , v ) \in \NCord (2n-2)$ 
obtained in this way will be called {\em the pair-parent} of $( \pi, u )$
and will be denoted as $\pairparent ( \pi , u )$.
\end{definition}

\vspace{6pt}

\begin{example}  \label{example:73}
Suppose that $n = 7$ and that we look at the pair-partition
\[
\pi = \bigl\{ \, \{1,6\}, \, \{2,3\}, \, \{4,5\}, \, \{7,8\},
\, \{9,14\}, \, \{10, 13\}, \, \{11,12\} \, \bigr\} \in NC_2 (14),
\]
considered with the monotonic ordering $u: \pi \to \{ 1, \ldots , 7 \}$
shown in the following picture:

\[ 
(\pi,u) = \begin{tikzpicture}[scale=0.5,baseline={2*height("$=$")}]
        \draw (0,0) -- (0,1.5) -- (5,1.5) -- (5,0);
        \node at (0.3,1.1) {$2$};
        \draw (1,0) -- (1,1) -- (2,1) -- (2,0);
        \node at (1.3,0.6) {$6$};
        \draw (3,0) -- (3,1) -- (4,1) -- (4,0);
        \node at (3.3,0.6) {$3$};
        \draw (6,0) -- (6,1) -- (7,1) -- (7,0);
        \node at (6.3,0.6) {$1$};
        \draw (8,0) -- (8,2) -- (13,2) -- (13,0);
        \node at (8.3,1.6) {$4$};
        \draw (9,0) -- (9,1.5) -- (12,1.5) -- (12,0);
        \node at (9.3,1.1) {$5$};
        \draw (10,0) -- (10,1) -- (11,1) -- (11,0);
        \node at (10.3,0.6) {$7$};
\end{tikzpicture} \in \NCord_2(14) \text{.} 
\]

\vspace{10pt}

\noindent
Then $J(\pi,u) = \{11,12\}$ and the pair-partition
$( \rho, v ) = \pairparent ( \pi, u ) \in \NCord_2 (12)$ is depicted like this:

\[ 
(\rho,v) = \begin{tikzpicture}[scale=0.5,baseline={2*height("$=$")}]
        \draw (0,0) -- (0,1.5) -- (5,1.5) -- (5,0);
        \node at (0.3,1.1) {$2$};
        \draw (1,0) -- (1,1) -- (2,1) -- (2,0);
        \node at (1.3,0.6) {$6$};
        \draw (3,0) -- (3,1) -- (4,1) -- (4,0);
        \node at (3.3,0.6) {$3$};
        \draw (6,0) -- (6,1) -- (7,1) -- (7,0);
        \node at (6.3,0.6) {$1$};
        \draw (8,0) -- (8,1.5) -- (11,1.5) -- (11,0);
        \node at (8.3,1.1) {$4$};
        \draw (9,0) -- (9,1) -- (10,1) -- (10,0);
        \node at (9.3,0.6) {$5$};
\end{tikzpicture} \in \NCord_2(12) \text{.} 
\]
\end{example}

\vspace{6pt}

\begin{notation-and-remark}  \label{def:74}
$1^o$ Let $n \geq 2$ be in $\bN$ and consider a $( \rho , v ) \in \NCord_2 (2n-2)$.
The set of children of $( \rho, v )$ in the $\NCord_2$-tree will be denoted as 
$\CPair ( \rho , v )$.  Thus $\CPair ( \rho , v ) \subseteq \NCord_2 (2n)$,
and upon examining the parent-child relation  
described in Definition \ref{def:72} one easily sees that:
\begin{equation}   \label{eqn:74a}
\left\{  \begin{array}{c}
\mbox{for every $1 \leq m \leq 2n-1$ there exists a 
      $( \pi_m , u_m ) \in \CPair ( \rho, v )$,}          \\
\mbox{uniquely determined, such that } J( \pi_m, u_m ) = \{ m, m+1 \}.
\end{array}   \right.
\end{equation}
As a consequence of (\ref{eqn:74a}), it follows that $| \CPair ( \rho, v ) | = 2n-1$.

\vspace{6pt}

\noindent
$2^o$ For the record, we note here that the homogeneity property (\ref{eqn:1-3b}) stated 
in Subsection \ref{subsection:1-3} of the Introduction is an immediate consequence of 
(\ref{eqn:74a}), and same for the further consequence stated in (\ref{eqn:1-3bb}), that 
$| \NCord_2 (2n) | = (2n-1)!!$ for every $n \in \bN$. 
\end{notation-and-remark}

\vspace{10pt}

\subsection{\texorpdfstring{\boldmath{$\NCord_2$}}{NCord2}-analogue for the recursive
pattern ``of the second kind''.}
\label{subsection:7-2}

$\ $

\noindent
For a random $( \pi, u ) \in \NCord_2 (2n)$, the statistic which counts the 
blocks of $\pi$ is trivial, and same is the case for the statistics which count 
blocks tallied by size.  It is, however, non-trivial to examine the statistic which 
counts the outer blocks of $\pi$.  In what follows we will see how this can be done 
by adjusting to $\NCord_2$-setting the discussion around the recursive condition of 
the second kind which was introduced in Subsection \ref{subsection:2-3} in connection 
to the full $\NCord$-tree. For starters, here is the adjusted version of that notion.

\vspace{6pt}

\begin{definition}  \label{def:75}
Let $\alpha,\beta, q \in \bZ$ be given.  We will say that a sequence 
of random variables $(Z_n : \NCord_2(2n) \to \bR)_{n=1}^{\infty}$ is 
\emph{pair-recursive of the second kind} with input $( \alpha , \beta ; q)$ when 
it has the following property: for every $n \geq 2$ and $(\rho,v) \in \NCord_2(2n-2)$, 
the set $\CPair ( \rho, v )$ of children of $( \rho , v )$ in the $\NCord_2$-tree 
has a subset $\CPair_o ( \rho, v )$ such that:
\begin{equation}   \label{eqn:75a}
\left\{   \begin{array}{ll}
\mathrm{(1)} & \mbox{If $( \pi, u ) \in \CPair_o ( \rho,v )$, then
      $Z_n(\pi,u) = Z_{n-1}(\rho,v) + \alpha$,}    \\ 
\mathrm{(2)} & \mbox{If $( \pi, u ) \in \CPair ( \rho, v ) \setminus \CPair_o ( \rho, v)$, 
                     then $Z_n(\pi,u) = Z_{n-1}(\rho,v) + \beta$,}    \\ 
\mbox{$\ $} \mathrm{ and} &                                        \\
\mathrm{(3)} & \mbox{$| \CPair_o ( \rho, v) | = Z_{n-1}(\rho,v)+q$.}
\end{array}   \right.
\end{equation}
\end{definition}

\vspace{6pt}

\begin{remark-and-notation}   \label{rem:76}
In analogy with what we saw in the full $\NCord$-setting, the recursion 
property from Definition \ref{def:75} entails a recursion for the 
combinatorial Laplace transforms of the random variables in question.  For any 
$n \in \bN$ and $Z_n : \NCord_2 (2n) \to \bR$, the combinatorial Laplace 
transform of $Z_n$ is the function $L_n : (0, \infty ) \to \bR$ defined by a 
recipe similar to the one used in the preceding sections,
\begin{equation}   \label{eqn:76a}
L_n (t) = \sum_{ ( \pi, u ) \in \NCord_2 (2n) } t^{Z_n ( \pi, u )},
\ \ t > 0.
\end{equation}
When $Z_n$ is viewed as a random variable with respect to the uniform distribution 
on $\NCord_2 (2n)$, the function $L_n$ can be used towards computing moments of $Z_n$,
and in particular the expectation of $Z_n$ is retrieved via the formula
\begin{equation}   \label{eqn:76b}
E[ Z_n ] = \frac{L_n ' (1)}{L_n (1)} = \frac{L_n ' (1)}{ (2n-1)!! }
\end{equation}
(similar to the one from the parallel discussion about random variables 
on $\NCord (n)$, but where we now use the normalization constant 
$L_n (1) = | \NCord_2 (2n) | = (2n-1)!!$).

The next theorem and corollary give the $\NCord_2$-analogue for the Theorem \ref{thm:61} 
and Corollary \ref{cor:62} of the preceding section.
The proofs are mutatis mutandis repetitions of arguments already shown in 
Section \ref{section:6}, and are left to the reader.
\end{remark-and-notation}

\vspace{6pt}

\begin{theorem}  \label{thm:77}
Let $\alpha, \beta, q \in \bZ$ and let $( Z_n : \NCord_2 (2n) \to \bR )_{n=1}^{\infty}$ be
a sequence of random variables which is pair-recursive of the second kind with input 
$( \alpha , \beta ; q )$.  For every $n \in \bN$, let $L_n : ( 0, \infty ) \to \bR$ be the 
combinatorial Laplace transform of $Z_n$.  Then one has the recursion:
\begin{equation}  \label{eqn:77a}
L_n (t) = \Bigl( q \, t^{\alpha} + ((2n-1)-q) \, t^{\beta} \Bigr) \cdot L_{n-1} (t)
          + \bigl( t^{\alpha + 1}-t^{\beta +1} \bigr) \cdot L_{n-1}'(t),
\end{equation}
holding for $n \geq 2$ and $t \in ( 0, \infty )$.
\hfill  $\square$
\end{theorem}

\vspace{6pt}

\begin{corollary}    \label{cor:78}
In the framework and notation of Theorem \ref{thm:77}, the sequence of expectations
$\bigl( \, E[ Z_n ] \, \bigr)_{n=1}^{\infty}$ satisfies the recursion
\begin{equation}  \label{eqn:78a}
E[ Z_n ] 
=  \frac{(2n-1) + ( \alpha - \beta)}{2n-1} \cdot E[ Z_{n-1} ] 
+ \frac{q \, \alpha + ((2n-1)-q) \, \beta}{2n-1},
\ \ n \geq 2.
\end{equation}
\hfill  $\square$
\end{corollary}

\vspace{10pt}

\subsection{Outer pairs and interval pairs -- proof of Theorem \ref{thm:110}.}
\label{subsection:7-3}

$\ $

\noindent
In this subsection we take on the random variables
$\bigl( \OuT_n : \NCord_2(2n) \to \bR \bigr)_{n=1}^{\infty}$
and $\bigl( \,\InT_n : \NCord_2(2n) \to \bR \, \bigr)_{n=1}^{\infty}$
that were 
brought to attention in Subsection \ref{subsection:1-3}, and we give the proof
of Theorem \ref{thm:110} concerning the expectations of these random variables.
The latter theorem will actually come out as an easy consequence of the 
considerations made in Subsection \ref{subsection:7-2}, due to the following 
lemma.

\vspace{6pt}

\begin{lemma}   \label{lemma:79}
$1^o$ The sequence of random variables
$\bigl( \InT_n : \NCord_2(2n) \to \bR \bigr)_{n=1}^{\infty}$ 
is pair-recursive of the second kind, with input $( 0, 1; 0)$.

\vspace{6pt}

\noindent
$2^o$ The sequence of random variables
$\bigl( \OuT_n : \NCord_2(2n) \to \bR \bigr)_{n=1}^{\infty}$ 
is pair-recursive of the second kind, with input $( 1, 0; 1)$.
\end{lemma}

\begin{proof}
$1^o$ Pick an $n \geq 2$ and a $( \rho , v ) \in \NCord_2 (2n-2)$, 
in connection to which we verify the conditions (1), (2), (3) stated in 
Definition \ref{def:75}.  Let $m$ be the number of interval pairs
of $\rho$, and let us explicitly denote these pairs as 
$\{ i_1, i_1 + 1 \}, \ldots , \{i_m, i_m + 1 \}$.

Now let us look at a $(\pi,u) \in \CPair (\rho, v) \subseteq \NCord_2(2n)$ (a child
of $(\rho, v )$ in the $\NCord_2$-tree), and let $V := J( \pi , u ) \in \pi$
be the interval pair which is appended to $\rho$ in order to go from 
$( \rho, v )$ to $( \pi , u )$.  A moment's thought shows that we have:
\begin{equation}   \label{eqn:79a}
\InT_n (\pi,u) = \left\{  \begin{array}{ll}
\InT_{n-1}(\rho,v), & \mbox{ if $V$ is inserted into one of the}     \\
                    & \mbox{$\ $ pairs $\{ i_1, i_1 + 1 \}, \ldots , 
                                      \{i_m, i_m + 1 \}$ of $\rho$,} \\
\InT_{n-1}(\rho,v) + 1,   & \mbox{ otherwise.}
\end{array}  \right.
\end{equation}
This verifies the conditions (1) and (2) from (\ref{eqn:75a}), 
with $\alpha = 0$ and $\beta = 1$.  We note, moreover, that the number of
possibilities counted in the first branch on the right-hand side of (\ref{eqn:79a}) 
is precisely equal to $\InT_{n-1} ( \rho , v )$ (which is our $m$); this shows 
that the condition (\ref{eqn:75a})(3) is also fulfilled, with $q=0$.

In order to illustrate the discussion made in the preceding paragraph, 
consider an example when $n=8$ and $( \rho , v ) \in \NCord_2 (14)$ has
underlying pair-partition
 \[ 
    \rho = \begin{tikzpicture}[scale=0.5,baseline={2*height("$=$")}]
        \draw (0,0) -- (0,1.5) -- (5,1.5) -- (5,0);
        \draw (1,0) -- (1,1) -- (2,1) -- (2,0);
        \draw (3,0) -- (3,1) -- (4,1) -- (4,0);
        \draw (6,0) -- (6,1) -- (7,1) -- (7,0);
        \draw (8,0) -- (8,2) -- (13,2) -- (13,0);
        \draw (9,0) -- (9,1.5) -- (12,1.5) -- (12,0);
        \draw (10,0) -- (10,1) -- (11,1) -- (11,0);
    \end{tikzpicture} \in NC_2(14)                    
\]
(that is,
$\rho = \bigl\{ \, \{1,6\}, \, \{2,3\}, \, \{4,5\}, \, \{7,8\}, 
\, \{9, 14\}, \, \{10,13\}, \, \{11,12\} \, \bigr\} \in NC_2 (14)$).
\noindent
This $\rho$ has $4$ interval pairs: $\{2,3\}, \, \{4,5\}, \, \{7,8\}, \, \{11,12\}$.  
The next picture shows, in dotted lines, how an interval pair ``$V = J( \pi, u )$'' 
can be inserted into one of these interval pairs of $\rho$. 
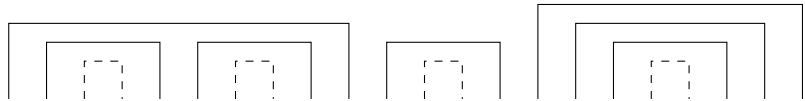
    \[ \begin{tikzpicture}[scale=0.5]
        \draw (0,0) -- (0,2) -- (9,2) -- (9,0);
        \draw (1,0) -- (1,1.5) -- (4,1.5) -- (4,0);
        \draw[dashed] (2,0) -- (2,1) -- (3,1) -- (3,0);
        \draw (5,0) -- (5,1.5) -- (8,1.5) -- (8,0);
        \draw[dashed] (6,0) -- (6,1) -- (7,1) -- (7,0);
        \draw (10,0) -- (10,1.5) -- (13,1.5) -- (13,0);
        \draw[dashed] (11,0) -- (11,1) -- (12,1) -- (12,0);
        \draw (14,0) -- (14,2.5) -- (21,2.5) -- (21,0);
        \draw (15,0) -- (15,2) -- (20,2) -- (20,0);
        \draw (16,0) -- (16,1.5) -- (19,1.5) -- (19,0);
        \draw[dashed] (17,0) -- (17,1) -- (18,1) -- (18,0);
    \end{tikzpicture} 
    \]
In each of the $4$ cases shown in dotted lines, the resulting 
$( \pi , u ) \in \NCord_2 (16)$ will have 
$\InT_8 ( \pi , u ) = \InT_7 ( \rho, v )$.  (For instance if the 
inserted pair is $V = \{ 3,4 \}$, then $V$ itself is an interval pair 
of $\pi$, but the former interval pair $\{ 2,3 \} \in \rho$ has now 
become $\{ 2,5 \} \in \pi$, and is lost from the tallying of interval 
pairs of $\pi$.)  On the other hand, if the inserted pair $V$ is not 
one of those shown in dotted lines in the picture, then the resulting
$( \pi, u ) \in \NCord_2 (8)$ will have 
$\InT_8 ( \pi , u ) = \InT_7 ( \rho, v ) + 1 = 5.$  This is because, 
in addition to $V$, the pair-partition $\pi$ will have $4$ interval pairs
that are inherited from $\rho$.

\vspace{6pt}

\noindent
$2^o$ The proof of this part goes on similar lines to the one for $1^o$, 
and is also very similar to the arguments presented in connection to 
Example \ref{example:212} earlier on.  Because of that, we will omit the 
details and only briefly explain how the requirements made in
Definition \ref{def:75} get to be fulfilled.  

So let us pick an $n \geq 2$ and a $( \rho , v ) \in \NCord_2 (2n-2)$.
Direct analysis, similar to the one from Example \ref{example:212}
shows that for $( \pi, u ) \in \CPair ( \rho, v )$ we have:
\begin{equation}   \label{eqn:79b}
\OuT_n (\pi,u) = \left\{  \begin{array}{ll}
\OuT_{n-1}(\rho,v)+1, & \mbox{ if $J( \pi,u)$ is an outer pair of $\pi$,} \\
\OuT_{n-1}(\rho,v),   & \mbox{ otherwise.}
\end{array}  \right.
\end{equation}
Thus the set 
\[
\CPair_o ( \rho , v ) := \{ ( \pi,u ) \in \CPair ( \rho , v ) \mid
\mbox{$J( \pi,u)$ is an outer pair of $\pi$} \}
\]
will fulfill conditions (1) and (2) from Definition \ref{def:75}, with 
$\alpha = 1$ and $\beta = 0$.
The remaining condition (3) amounts to the fact that, 
in (\ref{eqn:79b}), there are exactly $\OuT_{n-1}(\rho,v)+1$ choices of 
$( \pi , u )$ for which $J( \pi , u )$ is an outer pair of $\pi$.  This is 
indeed true, because an interval pair appended to $\rho$ is outer precisely
when it is inserted in between two consecutive outer pairs of $\rho$, or at 
one of the two ends (left or right) of the picture of $\rho$, and the number 
of ways to do this is equal to 
$\bigl( \OuT_{n-1}(\rho,v) - 1 \bigr) + 2 = \OuT_{n-1} ( \rho , v ) + 1$, 
as desired.

For illustration, consider again the example where $n=8$ and
$( \rho, v ) \in \NCord (14)$ has the same underlying $\rho \in NC_2 (14)$
used to illustrate the proof of part $1^o$ of the lemma.
There are $4$ possible ways to append to $\rho$ an interval pair which 
is outer in the resulting $\pi \in NC_2 (16)$; these $4$ possibilities are 
indicated with dotted lines in the next picture. 
    
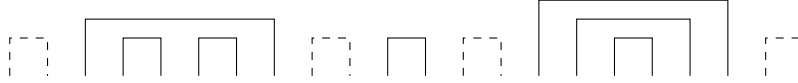
\[ \begin{tikzpicture}[scale=0.5]
        \draw[dashed] (0,0) -- (0,1) -- (1,1) -- (1,0);
        \draw (2,0) -- (2,1.5) -- (7,1.5) -- (7,0);
        \draw (3,0) -- (3,1) -- (4,1) -- (4,0);
        \draw (5,0) -- (5,1) -- (6,1) -- (6,0);
        \draw[dashed] (8,0) -- (8,1) -- (9,1) -- (9,0);
        \draw (10,0) -- (10,1) -- (11,1) -- (11,0);
        \draw[dashed] (12,0) -- (12,1) -- (13,1) -- (13,0);
        \draw (14,0) -- (14,2) -- (19,2) -- (19,0);
        \draw (15,0) -- (15,1.5) -- (18,1.5) -- (18,0);
        \draw (16,0) -- (16,1) -- (17,1) -- (17,0);
        \draw[dashed] (20,0) -- (20,1) -- (21,1) -- (21,0);
    \end{tikzpicture} 
    \]
As seen in the picture, the $4$ dotted possibilities correspond precisely
to the $2$ spaces between the $3$ outer pairs of $\rho$, plus the 
space to the left of $\{ 1,6 \}$ (leftmost outer pair of $\rho$)
and the one to the right of $\{ 9,14 \}$ (rightmost outer pair of $\rho$).
\end{proof}

\vspace{10pt}

\begin{ad-hoc-item}  \label{proof:710}
{\bf Proof of Theorem \ref{thm:110}.} 
$1^o$ In view of Lemma \ref{lemma:79}.1, we can invoke the recursion (\ref{eqn:78a}) 
found in Corollary \ref{cor:78}, where we put $\alpha = 0$, $\beta = 1$, $q =0$.
The said recursion becomes
\begin{equation}   \label{eqn:710a}
E[ \InT_n ] = \frac{2n-2}{2n-1} \cdot E[ \InT_{n-1} ] + 1, \ \ n \geq 2;
\end{equation}
from here, an easy induction verifies the claim that 
$E[ \InT_n ] = (2n+1)/3$ for all $n \in \bN$.

\vspace{6pt}

\noindent
$2^o$ In the same vein as in $1^o$ above, we now use the recursion (\ref{eqn:78a})
with specialization $\alpha = 1$, $\beta = 0$, $q =1$, and this gives:
\begin{equation}   \label{eqn:710b}
E[ \OuT_n ] = \frac{2n}{2n-1} \cdot E[ \OuT_{n-1} ] 
+ \frac{1}{2n-1}, \ \ n \geq 2.
\end{equation}
When examining the formula (\ref{eqn:710b}), it is useful to observe that it can 
be rewritten as
\[
E[ \OuT_n ] +1  = \frac{2n}{2n-1} \cdot \bigl( \, E[ \OuT_{n-1} ] + 1 \, \bigr), 
\ \ n \geq 2;
\]
from here, an immediate induction establishes the fact that for every $n \in \bN$
we have 
\[
E [ \OuT_n ] + 1 = \frac{2 \cdot 4 \cdots (2n)}{1 \cdot 3 \cdots (2n-1)}
= \frac{2^n \cdot n!}{ (2n-1)!! },
\]
which leads to the formula (\ref{eqn:110a}) in the statement of
Theorem \ref{thm:110}.2.
The asymptotics $\bE[ \OuT_n ] \sim \sqrt{\pi n}$ for $n \to \infty$ follows
from (\ref{eqn:110a}) via an easy application of Stirling's approximation formula.
\hfill $\square$
\end{ad-hoc-item} 

\vspace{0.5cm}
 
\section{Expectation of the ``area'' statistic on 
\texorpdfstring{$\NCord_2 (2n)$}{NCord2(2n)}}
\label{section:8}

\setcounter{equation}{0}

In this section we consider the framework from Section \ref{subsection:1-4} of 
the Introduction, and we give the proof of Theorem \ref{thm:111}, concerning the 
expectation of the area-measuring random variable $\Area_n : \NCord_2 (2n) \to \bR$.
Recall that the latter random variable was introduced by putting 
$\Area_n ( \pi, u ) := \area ( \pi )$, the area in between the horizontal axis and 
the Dyck path naturally associated to $\pi$.  In the next lemma we give an alternate
description for what is $\area ( \pi )$, which does not refer to Dyck paths.

\vspace{6pt}

\begin{lemma}   \label{lemma:81}
Let $n$ be in $\bN$.  One has that:
\[
( \, \diamondsuit_n \, )  \hspace{2cm}
\area ( \pi ) = \sum_{V \in \pi}  \max (V) - \min (V),
\ \ \forall \, \pi \in NC_2 (2n).
\hspace{2cm}  \mbox{$\ $}
\]
\end{lemma}

\begin{proof}  By induction on $n$.  The statement $( \, \diamondsuit_1 \, )$ 
clearly holds, since the unique pair-partition $\pi = \{ V_1 \} \in NC_2 (2)$ has 
$\area ( \pi ) = 1 = \max (V_1) - \min (V_1)$.  For the rest of the proof we 
fix an $n \geq 1$, we assume that the statement $( \, \diamondsuit_m \, )$ was 
proved for every $1 \leq m \leq n$, and we verify that 
$( \, \diamondsuit_{n+1} \, )$ holds as well.  We divide the verification into
two cases.

\vspace{6pt}

\noindent
{\em Case 1. Consider a $\pi \in NC_2 (2n+2)$ such that $\{1, 2n+2 \}$ is a pair
of $\pi$.}

\vspace{3pt}

\noindent
In this case, we let $\pi_o \in NC_2 (2n)$ be the relabeled restriction of 
$\pi$ to $\{ 2,3, \ldots , 2n+1 \}$. If we write explicitly
$\pi = \{V_1, \ldots , V_n, V_{n+1} \}$ with $V_{n+1} = \{1, 2n+2\}$, then 
$\pi_o$ can be written as $\{V_1', \ldots , V_n'\}$ with 
$V_j'$ obtained by translating $V_j$ by 1 to the left (and thus having 
$\max (V_j') - \min (V_j') =  \max (V_j) - \min (V_j)$), for $1 \leq j \leq n$.
On the right-hand side of the formula $( \, \diamondsuit_{n+1} \, )$ we get:
\[
\sum_{V \in \pi} \max (V) - \min (V)
= \bigl( \, \sum_{j=1}^n \max (V_j) - \min (V_j) \, \bigr) 
            + \bigl( \, (2n+2) - 1 \, \bigr)      
\]
\[
= \bigl( \, \sum_{j=1}^n \max (V_j') - \min (V_j') \, \bigr) + (2n+1) 
= \area ( \pi_o ) + (2n+1),
\]
where at the third equality sign we invoked the induction hypothesis.

In order to complete the verification of this case, we are left to check that
$\area ( \pi ) = \area ( \pi_o ) + (2n+1)$.  But this is immediately seen by 
looking at the Dyck paths associated to $\pi$ and $\pi_o$.
Indeed, the path for $\pi$ can be obtained by taking the path for $\pi_o$ and placing 
it on top of the trapezoid with vertices at the points $(0,0)$, 
$(1,1), (2n+1,1)$ 
and $(2n+2, 0)$ of $\bZ^2$, where the latter trapezoid has area equal to $2n+1$.

\vspace{6pt}

\noindent
{\em Case 2. Consider a $\pi \in NC_2 (2n+2)$ where $1$ is paired 
with $2k$ for some $1 \leq k \leq n$.}

\vspace{3pt}

\noindent
In this case, we let $\pi_1 \in NC_2 (2k)$ and $\pi_2 \in NC_2 ( \, 2(n+1-k) \, )$ 
be the relabeled restrictions of $\pi$ to $\{1, \ldots , 2k \}$ and to
$\{ 2k+1, \ldots , 2n+2 \}$, respectively.  Upon looking at the pictures of the 
Dyck paths associated to $\pi, \pi_1$ and $\pi_2$, one immediately sees that
\begin{equation}   \label{eqn:81b}
\area ( \pi) = \area ( \pi_1 ) + \area ( \pi_2 ).
\end{equation}
The induction hypothesis applies to both $\pi_1$ and $\pi_2$, hence 
(\ref{eqn:81b}) can be continued with 
\[
= \Bigl( \, \sum_{V' \in \pi_1} \max (V') - \min (V') \, \Bigr)
+ \Bigl( \, \sum_{V'' \in \pi_2} \max (V'') - \min (V'') \, \Bigr),
\]
which is immediately seen to just be 
$\sum_{V \in \pi} \max (V) - \min (V)$, as required.
\end{proof}

\vspace{6pt}

In order to approach the area-measuring random variables $\Area_n$ by using 
the $\NCord_2$-tree, we prove another lemma, as follows.

\vspace{6pt}

\begin{lemma}   \label{lemma:82}
Let $n \geq 2$ be in $\bN$, let $( \rho , v )$ be in $\NCord_2 (2n-2)$,
and consider the set $\CPair ( \rho, v ) \subseteq \NCord_2 (2n)$ consisting 
of children of $( \rho, v )$ in the $\NCord_2$-tree 
(cf.~Notation and Remark \ref{def:74}).  One has that
\begin{equation}  \label{eqn:82a}
\sum_{ ( \pi, u ) \in \CPair ( \rho , v ) } \Area_n ( \pi, u )
= (2n-1) + (2n+1) \cdot \Area_{n-1} ( \rho, v ).
\end{equation}
\end{lemma}

\begin{proof} As observed in Notation and Remark \ref{def:74}, we can spell out
\[
\CPair ( \rho, v ) = \bigl\{ ( \pi_1, u_1 ), \ldots , ( \pi_{2n-1}, u_{2n-1} ) \bigr\},
\]
where for every $1 \leq m \leq 2n-1$ the ordered partition 
$( \pi_m , u_m ) \in \NCord_2 (2n)$ is uniquely determined by the requirements 
$\pairparent ( \pi_m , u_m ) = ( \rho , v )$ and $J( \pi_m, u_m ) = \{ m, m+1 \}$.
For further processing, let us fix an explicit writing $\rho = \{ V_1, \ldots , V_{n-1} \}$.
Then for every $1 \leq m \leq 2n-1$ we get an explicit writing
$\pi_m = \bigl\{ V_1^{(m)}, \ldots , V_{n-1}^{(m)}, \, \{m, m+1\} \, \bigr\},$
where $V_1^{(m)}, \ldots , V_{n-1}^{(m)}$, are the images of 
$V_1, \ldots , V_{n-1}$, respectively, under the unique order-preserving bijection 
from $\{ 1, \ldots , 2n-2 \}$ onto $\{ 1, \ldots , 2n \} \setminus \{m, m+1 \}$.
This gives:
\[
\sum_{ ( \pi,u ) \in \CPair ( \rho , v ) } \Area_n ( \pi,u)
= \sum_{m=1}^{2n-1} \mathrm{area} ( \pi_m ) \hspace{4.5cm}  \mbox{$\ $}                  
\]
\[ 
\hspace{2cm} 
= \sum_{m=1}^{2n-1} \Bigl( \, 1 + \sum_{j=1}^{n-1} 
               \max ( V_j^{(m)} ) - \min ( V_j^{(m)} ) \, \Bigr)  
               \ \mbox{ (by Lemma \ref{lemma:81}) }
\]
\begin{equation}  \label{eqn:82b}
= (2n-1) + \sum_{j=1}^{n-1} \Bigl( \, \sum_{m=1}^{2n-1} 
               \max ( V_j^{(m)} ) - \min ( V_j^{(m)} ) \, \Bigr).
\end{equation}

Let us now fix for the moment a $j \in \{ 1, \ldots , n-1 \}$, and let us examine the 
inside sum over $m$ in (\ref{eqn:82b}).  For every $1 \leq m \leq 2n-1$, the relation 
between the pair $V_j^{(m)} \in \pi_m$ and the pair $V_j \in \pi$ forces the quantity
$\max ( V_j^{(m)} ) - \min ( V_j^{(m)} )$ to be equal to
\begin{equation}   \label{eqn:82c}
\left\{   \begin{array}{l}
\bigl[ \, \max (V_j) - \min (V_j) \, \bigr] + 2, 
    \mbox{ if $\min (V_j) \leq m-2$ and $\max (V_j) \geq m -1$,}       \\
                      \\
\bigl[ \, \max (V_j) - \min (V_j) \, \bigr], \ \mbox{ otherwise.}   
\end{array}   \right.
\end{equation}
Summing over $m$ in (\ref{eqn:82c}) thus gives
\begin{align*}
\sum_{m=1}^{2n-1} \max ( V_j^{(m)} ) - \min ( V_j^{(m)} )
& = (2n-1) \cdot \bigl[ \,  \max (V_j) - \min (V_j) \, \bigr]
     + \sum_{m = \min ( V_j ) + 2}^{ \max (V_j) + 1 } 2   \\
& = (2n-1) \cdot \bigl[ \,  \max (V_j) - \min (V_j) \, \bigr]
     + ( \max (V_j) - \min (V_j) ) \cdot 2       \\
& = (2n+1) \cdot \bigl[ \, \max (V_j) - \min (V_j) \, \bigr],
\end{align*}
where at the second equality sign we observed that the summation range for 
$m$, on the preceding line, has $\max (V_j) - \min (V_j)$ terms.

We now unfix the index $j$ and replace the result of the preceding calculation
into the sum over $j$ from (\ref{eqn:82b}), 
in order to obtain the required formula (\ref{eqn:82a}):
\[
\sum_{ ( \pi,u ) \in \CPair ( \rho , v ) } \Area_n ( \pi,u)
 = (2n-1) + \sum_{j=1}^{n-1} \Bigl( \, 
 (2n+1) \cdot \bigl[ \, \max (V_j) - \min (V_j) \, \bigr] \, \Bigr)
 \]
 \[
 = (2n-1) + (2n+1) \cdot \Area_{n-1} ( \rho, v ),
 \]
 where at the latter equality sign we used Lemma \ref{lemma:81} in connection 
 to $( \rho , v )$.
\end{proof} 

\vspace{10pt}

\begin{ad-hoc-item}    \label{proof:83}
{\bf Proof of Theorem \ref{thm:111}.}
For every $n \in \bN$, let us put
\begin{equation}   \label{eqn:83a}
S_n := \sum_{ ( \pi, u ) \in \NCord (n) } A_n ( \pi, u ).
\end{equation}
We claim that the numbers $S_n$ satisfy the following recursion:
\begin{equation}   \label{eqn:83b}
S_n = (2n+1) \, S_{n-1} + (2n-1)!! \, , \ \ \forall \, n \geq 2.
\end{equation}
In order to prove (\ref{eqn:83b}), we do the usual trick of grouping the 
terms on the right-hand side of (\ref{eqn:83a}) according to what is the 
parent of $( \pi, u )$ in the $\NCord_2$-tree:
\begin{align*}
S_n 
& = \sum_{ ( \rho, v ) \in \NCord_2 (n-1) } 
\, \Bigl( 
\sum_{ ( \pi, u ) \in \CPair (\rho, v ) } A_n ( \pi, u ) \, \Bigr),   \\
& = \sum_{ ( \rho, v ) \in \NCord_2 (n-1) } 
\, \Bigl( (2n-1) + (2n+1) \cdot \Area_{n-1} ( \rho, v ) \, \Bigr)
          \mbox{ (by Lemma \ref{lemma:82})}                           \\
& = (2n-1) \cdot | \, \NCord_2 (n-1) \, | 
+ (2n+1) \cdot \sum_{ ( \rho, v ) \in \NCord_2 (n-1) } 
 \Area_{n-1} ( \rho, v )                                    \\
&  = (2n-1) \cdot (2n-3)!! +  (2n+1) \cdot S_{n-1},
\end{align*}
which gives the right-hand side of (\ref{eqn:83b}).

From (\ref{eqn:83b}), an easy induction on $n$ yields the exact formula
\begin{equation}   \label{eqn:83c}
S_n = (2n+1)!! \, \Bigl( \, \frac{1}{3} + \frac{1}{5} 
+ \cdots + \frac{1}{2n+1} \, \Bigr), \ \ n \geq 1 \text{.}
\end{equation}
Hence the formula for the expectation of $\Area_n$ becomes
\[
E[ \, \Area_n \, ] = \frac{S_n}{(2n-1)!!}
= \Bigl( \, (2n+1)!! \, \sum_{k=1}^n \frac{1}{2k+1} \, \Bigr) / (2n-1)!!
= (2n+1) \, \sum_{k=1}^n \frac{1}{2k+1},
\]
as claimed in Equation (\ref{eqn:111a}) from the statement of the theorem.

In order to verify that $E[ \, \Area_n \, ] \sim n \, \ln n$,
we rewrite the above formula as
\begin{equation}   \label{eqn:83d}
E[ \, \Area_n \, ] =
\bigl( n + \frac{1}{2} \bigr) \, \sum_{k=1}^n \frac{1}{k + \frac{1}{2}},
\end{equation}
where the sum on the right-hand side is bounded above by $\sum_{k=1}^n 1/k = H_n$ 
and bounded below by $\sum_{k=1}^n 1/(k+1) > H_n - 1$.  This leads to the inequalities
\[
\frac{n \, ( H_n - 1)}{n \, \ln n}
\leq \frac{E[ \, \Area_n \, ]}{n \, \ln n} \leq
\frac{(n+1) \, H_n}{n \, \ln n},
\ \ n \geq 2,
\]
and implies via squeeze that 
$\lim_{n \to \infty} E[ \, \Area_n \, ] / (n \, \ln n) = 1$.
\hfill $\square$
\end{ad-hoc-item}


\vspace{10pt}

\begin{thebibliography}{99}

\bibitem{Ar2012} O. Arizmendi.
Statistics of blocks in $k$-divisible non-crossing partitions, 
{\em Electronic Journal of Combinatorics} 19 (2012).

\vspace{6pt}

\bibitem{ArHaLeVa2015} O. Arizmendi, T. Hasebe, F. Lehner, C. Vargas. 
Relations between cumulants in non-commutative probability, 
{\em Advances in Mathematics} 282 (2015), 56-92. Also available as arXiv:1408.2977.

\vspace{6pt}

\bibitem{Be2005} A.C.R. Belton. 
A note on vacuum-adapted semimartingales and monotone independence,
in: {\em Quantum Probability and Infinite Dimensional Analysis XVIII. From Foundations to
Applications}, M. Schurmann and U. Franz editors. Published by World Scientific, 2005, 
pages 105–114.

\vspace{6pt}

\bibitem{Be2006} A.C.R. Belton.
The monotone Poisson process,
in: {\em Quantum Probability}, Banach Center Publications Volume 73.
Published by the Institute of Mathematics of the Polish Academy of Sciences,
2006, pages 99-115.

\vspace{6pt}

\bibitem{BoOl2017} A. Borodin, G. Olshanski.
{\em Representations of the infinite symmetric group},
Cambridge University Press, 2017.

\vspace{6pt}

\bibitem{HaSa2011} T. Hasebe, H. Saigo.
The monotone cumulants, {\em Annales Institute 
Henri Poincar\'e, Probabilit\'es et Statistiques} 47 (2011), 1160–1170.

\vspace{6pt}

\bibitem{Ka2020} V. Kargin.
Limit theorems for statistics of non-crossing partitions, arXiv:1907.00632.

\vspace{6pt}

\bibitem{MeSpVe1996} D. Merlini, R. Sprugnoli, M.C. Verri. 
The area determined by under-diagonal lattice paths, in {\em Proceedings of CAAP'96}.
Published by Springer as {\em Lecture Notes in Computer Science} Volume 1059, 1996, 
pages 59-71. 

\vspace{6pt}

\bibitem{Mu2001} N. Muraki.
Monotonic independence, monotonic central limit theorem and monotonic 
law of small numbers, 
{\em Infinite Dimensional Analysis and Quantum Probability} 4 (2001), 39-58.

\vspace{6pt}

\bibitem{NiSp2006} A. Nica, R. Speicher.
{\em Lectures on the combinatorics of free probability}, 
Cambridge University Press, 2006.

\vspace{6pt}

\bibitem{Or2012} J. Ortmann.
Large deviations for non-crossing partitions, 
{\em Electronic Journal of Probability}, 17 (2012).


\end{thebibliography}
\end{document}